\documentclass[12pt,fullpage]{amsart}
\usepackage{amssymb, amsfonts, latexsym, amsthm, amsmath,  framed, lipsum}
\usepackage[mathscr]{euscript}
\usepackage{graphicx}
\usepackage{url}
\usepackage{tkz-graph}
\usepackage{tikz-cd}
\usepackage{tkz-berge}
\usepackage[margin=1in]{geometry}

\usepackage[utf8]{inputenc}
\usepackage[T1]{fontenc}
\usepackage[english]{babel}
\usepackage{textcomp}
\usepackage{enumerate}

\usepackage{blindtext} 

\usepackage{scrlayer-scrpage}

\newcommand{\Q}{\mathbb{Q}}

\newcommand{\R}{\mathbb{R}}
\newcommand{\N}{\mathbb{N}}
\newcommand{\Z}{\mathbb{Z}}
\newcommand{\F}{\mathbb{F}}

\newcommand{\cRing}{\underline{\textbf{CRing}}}
\newcommand{\redPo}{\textbf{POR/$\sqrt{0}$}}
\newcommand{\lorf}{\mathscr{L}_{\leq}}
\newcommand{\rcr}{\textbf{RCR}}
\newcommand{\rce}{\textbf{rc\'{e}}}

\DeclareMathOperator{\im}{im}

\DeclareMathOperator{\colim}{colim}

\DeclareMathOperator{\supp}{supp}

\DeclareMathOperator{\sper}{Sper}
\DeclareMathOperator{\semf}{\text{semf}}

\theoremstyle{plain}
\newtheorem{thm}{Theorem}
\newtheorem{lemma}[thm]{Lemma}

\newtheorem{prop}[thm]{Proposition}
\newtheorem{cor}[thm]{Corollary}

\theoremstyle{definition}
\newtheorem{defn}[thm]{Definition}
\newtheorem{ex}[thm]{Example}

\newtheorem{remark}[thm]{Remark}

\theoremstyle{remark}

\begin{document}

\begin{center}
\copyright Copyright [2024] \\
\vspace{0.3cm}
Tafari Clarke-James 
\end{center}

\newpage

\begin{titlepage}
   \begin{center}
       \vspace{1cm}

       \textbf{Towards Cohomology of Real Closed Spaces}

       \vspace{0.5cm}
       Tafari Clarke-James
            
       \vspace{1.5cm}
       A dissertation\\ submitted in partial fulfilment of\\ the requirements for the degree of\\
       \vspace{1cm}
       Doctor of Philosophy

       University of Washington - Seattle\\
       2024
       
       Dissertation Committee: \\
       S\'{a}ndor Kov\'{a}cs \\
       Max Lieblich \\
       Kyle Ormsby



       \vfill
            
       Program Authorized to Offer Degree:\\ Theoretical Mathematics
            
       \vspace{0.8cm}

   \end{center}
\end{titlepage}

\newpage

\section*{University of Washington}
\begin{center}
\textbf{Abstract}\\
\vspace{0.5cm}
Towards Cohomology of Real Closed Spaces\\
\vspace{1.5cm}
Chair of Supervisory Committee: \\ S\'{a}ndor Kov\'{a}cs \\ Theoretical Mathematics\\ 
\vspace{2.5cm}
\flushleft It was shown by Claus Scheiderer prior to 1994 that real closed spaces have \'{e}tale cohomology. Following Scheiderer, study of real closed spaces fell out of fashion and o-minimal geometry became the focus for those at the intersection of model theory and geometry. I decided to breathe new life into the theory of real closed rings and spaces, as studied by Schwartz in 1989. In Section 1, I build the fundamentals of the theory using as little machinery as possible, and presented them as clearly as I could. Hidden gems include a full proof that real closed rings are closed under limits and colimits. In Section 2, I give an introduction to the category of real closed spaces in the first half. In the second half, I construct an equivalence of topoi between Scheiderer's sheaves on the real \'{e}tale site, and sheaves on a real \'{e}tale site $\rce/X$ of my creation. Since $\text{Sh}(\rce/X)$ can be defined without the use of $G$-topoi, the equivalence of topoi renders Scheiderer's theory computable. I end with a discussion of how one might use motivic cohomology to better understand recent results of Annette Huber in \cite{no_deRham_huber}. 
\end{center}
\vspace{4cm}

\newpage

\tableofcontents

\newpage

\section{ Sper(-) and Sections}

\subsection{Orders=Cones}

The Goal of this section is to introduce the ideas necessary to define $\text{Sper}(R)$, for $R$ a commutative ring with $1 \in R$, as a topological space with both the weak (read: Zariski) topology and the pro-constructible topology, together with sheaves of compatible and constructible sections. We assume some background in commutative algebra and point set topology, though quite minimal amounts of each.

Let $R$ be a commutative ring with $1 \in R$.

\begin{defn}
A \underline{positive cone of a partial order} $\alpha$ on $R$ is a subset $\alpha \subset R$ such that
\begin{enumerate}
    \item $0 \in \alpha$,
    \item $+_{R}:\alpha \times \alpha \to \alpha$ is a function,
    \item $*_{R}:\alpha \times \alpha \to \alpha$ is a function, and
    \item  $\forall x \in \alpha$, either $x=0$ or $-x \notin \alpha$.
\end{enumerate}
$\alpha$ is the \underline{positive cone of a total order} if we additionally require that
\[ \forall x \in R, x=0 \bigvee \left((x \in \alpha)\bigwedge (-x \notin \alpha)\right) \bigvee \left((-x \in \alpha)\bigwedge (x \notin \alpha)\right). \]
\end{defn}

We say $\alpha \subset R$ \textit{partially orders} $R$ if $\alpha$ is the positive cone of a partial order on $R$, and we say $\alpha$ \textit{totally orders} $R$ if $\alpha$ is the positive cone of a total order on $R$.

We observe that if a subset $\alpha$ of a ring $R$ is the positive cone of a total order on $R$, then $\alpha$ satisfies conditions (1)-(3), meaning $\alpha$ is a commutative semiring and a sub-semiring of $R$. An example of a positive cone of a total order is \[\Z_{\geq 0}:=\{x \in \Z:x \geq 0\}\] for the commutative ring $\Z$.

Reflecting further on $\Z$ with positive cone (short for positive cone of total order) $\Z_{\geq 0}$, we define a binary relation $\prec$ on $\Z$ as follows; for $x,y \in \Z$
\[ x \prec y \qquad \iff \qquad y-x \in \Z_{\geq 0}.  \]
If $x=y$, $y-x=0$, so by (1) in our definition of positive cones, $x \prec x$. If $x \prec y$ and $y \neq x$, then $y-x \in \alpha=\Z_{\geq 0}$, and consequently
\[ -(y-x)=x-y \notin \alpha \]
and $\prec$ is antisymmetric. Finally, if $x \prec y$ and $y \prec z$, 
\[ y-x,z-y \in \alpha=\Z_{\geq 0}, \]
so by property (2), \[(z-y)+(y-x)=z-x \in \alpha=\Z_{\geq 0},\] implying $x \prec z$ and that $\prec$ is transitive.

Thus a positive cone of a total order $\alpha \subset R$ does give a total order $\prec$ on the ring $R$. On the other hand, given a ring $R$ with total order $\prec$,
\[  \{ x \in R: 0 \prec x \} \]
is closed under addition and multiplication by properties of ordered rings, contains 0, and contains no pair $\{y,-y \}$ for $y \in R$, so it \textbf{is} the positive cone of its total order.

We need two more definition which will be handy later;

\begin{defn}
(\cite{convId},1) Let $R$ be a commutative unital ring with positive cone of total order $\alpha$. An ideal $I \leq R$ is $\alpha$-\textit{convex}  if the following sentence holds for all $a,b \in R$;
\begin{center}
If $a,b-a \in \alpha$, and $b \in I$, then $a \in I$.
\end{center}
\end{defn}
 The definition above explains how we can define a notion of convexity in an arbitrary commutative ring.
 
\begin{defn}
(\cite{rcs},2) Let $R$ be a commutative unital ring, with positive cone of partial order $\alpha \subset R$. The \underline{support} of $\alpha$ is the set
\[ \supp(\alpha):=\alpha \cap -\alpha \]
\end{defn}

The support $\supp(\alpha)$ of a positive cone of partial order $\alpha \subset R$ is closed under addition; for $r_1,r_2 \in \supp(\alpha)$, we have $r_{i}=-s_{i}$ for some $s_i \in \alpha$, and consequently
\[ -s_1 + -s_2=-(s_1 + s_2) \in -\alpha. \]
$\supp(\alpha)$ is similarly closed under multiplication because $(-1)^{2}=1$, and $0 \in \supp(\alpha)$. 

\begin{thm}
If $R$ is a ring with positive cone of partial order $\alpha$, then $\supp(\alpha)$ is an ideal in $R$.
\end{thm}
\begin{proof} Consider the function
\[ *_{R}:R \times \supp(\alpha) \to R. \]
We may decompose $R$ as the disjoint union
\[ R=(\alpha \setminus \{0\}) \cup -\alpha, \]
and we may consider the restrictions
\[ *_{R}|_{\alpha \setminus \{0\}}:(\alpha \setminus \{0\})\times \supp(\alpha) \to R \quad \& \quad *_{R}|_{-\alpha}:-\alpha \times \supp(\alpha) \to R. \]
We know
\begin{align*}
 (\alpha \setminus \{0\})*\supp(\alpha)&=(\alpha \setminus \{0\})*(\alpha \cap -\alpha) \\&=((\alpha \setminus \{0\})*\alpha)\cap ((\alpha \setminus \{0\})*-\alpha),  
\end{align*}
and since $(\alpha \setminus \{0\})*\alpha \subseteq \alpha$ and $(\alpha \setminus \{0\})*-\alpha \subseteq -\alpha$, we may restrict to obtain a multiplication
\[ *_{R}|_{\alpha \setminus \{0\}}:(\alpha \setminus \{0\})\times \supp(\alpha) \to \supp(\alpha). \]
An analogous argument shows we may restrict the codomain of 
\[ *_{R}|_{-\alpha}:-\alpha \times \supp(\alpha) \to R \]
to obtain a multiplication
\[ *_{R}|_{-\alpha}:-\alpha \times \supp(\alpha) \to \supp(\alpha),  \]
and by extension of maps of sets we obtain a multiplication
\[ *_{R}:R \times \supp(\alpha) \to \supp(\alpha).  \]
The above shows, for the positive cone of a total order $\alpha$, that $\supp(\alpha)$ is an ideal of $R$. 
\end{proof}

\subsection{Cones and Homomorphisms}

\begin{thm}
Let $f:A \to B$ be a homomorphism of rings, and let $\beta$ be the positive cone of a partial order on $B$. Then $\alpha:=f^{-1}(\beta)$ is the positive cone of a partial order on $A$.
\end{thm}
\begin{proof} 
We know $0 \in \alpha$ because ring homomorphisms map 0 to 0. Let $a_1,a_2 \in \alpha$. We have 
\begin{align*}
a_1+a_2 \in f^{-1}(\beta) &\iff f(a_1+a_2) \in \beta \\ &\iff f(a_1)+f(a_2) \in \beta.  
\end{align*}
Since $a_1,a_2 \in f^{-1}(\beta)$, $f(a_1) \, \& \, f(a_2)$ are in $ \beta$, as is $f(a_1)+f(a_2)$. We have thus showed $+_{A}:\alpha \times \alpha \to \alpha$ is a function, and the same argument shows $*_{A}:\alpha \times \alpha \to \alpha$ is a function. Lastly, 
\begin{align*}
 a_1 \in \alpha &\iff f(a_1) \in \beta \\ &\implies -f(a_1) \notin \beta \\ &\iff f(-a_1) \notin \beta \\ &\iff -a_1 \notin f^{-1}(\beta)=\alpha,
\end{align*}
and we've demonstrated positive cones of partial orders pull back to positive cones of partial orders under ring homomorphisms. 
\end{proof}

Now, we might ask whether the positive cones of \textit{total} orders on $B$ are pulled back to positive cones of total orders under ring homomorphisms $f:A \to B$. In general, they are not; 

\begin{ex}
Consider the $\R$-algebra homomorphism $f:\R[t] \to \R$ which maps $t \mapsto 3$. $\R$ is totally ordered with positive cone $\R_{\geq 0}$, and
\[ f^{-1}(\R_{\geq 0})=\R_{\geq 0}[t]. \]
If $\R_{\geq 0}[t]$ is the positive cone of a \textit{total} order, any two elements of $\R[t]$ must be comparable under that order. For the polynomials $t^{2}+1$ and $2t^{2}$,
\[ (t^{2}+1)-(2t^{2})=1-t^{2} \notin \R_{\geq 0}[t], \]
and 
\[  (2t^{2})-(t^{2}+1)=-(1-t^{2})=t^{2}-1 \notin \R_{\geq 0}[t], \]
so the existence of this pair demonstrates that the total order $\geq$ on $\R$ only pulls back to a partial order on $\R[t]$ via $f$.
\end{ex}

\begin{lemma}
If $f:A \to B$ is an injective ring homomorphism, and $\beta$ is the positive cone of a total order on $B$, $f^{-1}(\beta)$ is the positive cone of a total order on $A$.
\end{lemma}
\begin{proof}
 We already know $f^{-1}(\beta)$ is the total cone of a partial order on $A$. Fix a nonzero $a \in A$, and suppose $-a \notin f^{-1}(\beta)$. We want to show $a \in f^{-1}(\beta)$  We then know $f(-a)=-f(a) \notin \beta$, and consequently $f(a) \in \beta$. Since $f$ is injective, $f(a) \neq 0$. If $f(a) \in \beta$, $a \in f^{-1}(\beta)$ by definition, proving $f^{-1}(\beta)$ is the positive cone of a total order on $A$. 
 \end{proof}

\subsection{Introducing Pre-Geometry}

Let $f:A \to B$ be a ring homomorphism, and let $\beta \subset B$ such that $\beta/\supp(\beta)$ is the positive cone of a total order on $B/\supp(\beta)$, and $\supp(\beta)$ is prime. We will see shortly that subsets $\beta \subset B$ with prime support, such that $\beta/\supp(\beta) \subset B/\supp(\beta)$ totally orders $B/\supp(\beta)$, are the right objects to form a locally ringed space from! Since the inverse image of a prime ideal is prime under ring homomorphism, we obtain a commutative diagram of rings 
\begin{center}
\begin{tikzcd}
A \ar[r,"f"] \ar[d] & B \ar[d] \\ A/f^{-1}(\supp(\beta)) \ar[r,dotted,"\exists \bar{f}"] & B/\supp(\beta).
\end{tikzcd}
\end{center}
Since 
\begin{align*}
f^{-1}(\beta \cap -\beta)&=f^{-1}(\beta) \cap f^{-1}(-\beta)\\&= f^{-1}(\beta) \cap -f^{-1}(\beta),
\end{align*}
the support of the cone  of partial order $f^{-1}(\beta)$ is the prime ideal $f^{-1}(\supp(\beta))$, and we may extend the diagram above to
\begin{center}
\begin{tikzcd}
& A \ar[r,"f"] \ar[d] \ar[ld] & B \ar[d] \\ A/\supp(f^{-1}(\beta)) \ar[r,"\cong"]  & A/f^{-1}(\supp(\beta)) \ar[r,dotted,"\exists \bar{f}"] & B/\supp(\beta).
\end{tikzcd}
\end{center}

Let $\overline{a}$ denote the image of an element $a\in A$ under projection, and let $\overline{\alpha}$ denote the image of a subset $\alpha \subset A$ under projection (and analogously for $B$). I claim $\overline{f^{-1}(\beta)}$ is the positive cone of a total order on $A/\supp(f^{-1}(\beta))$. Suppose $\bar{a} \in A/\supp(f^{-1}(\beta))$ with $-\bar{a} \notin \overline{f^{-1}(\beta)}$. We then have 
\[\bar{f}(-\bar{a}) \notin \bar{\beta}, \]
and since $\bar{\beta}$ totally orders $B/\supp(\beta)$, 
\[ \bar{f}(\bar{a}) \in \bar{\beta}.  \]
By commutativity of the above diargram, $f(a) \in \beta$, and so 
\[ a \in f^{-1}(\beta). \]
Taking equivalence classes, $\bar{a} \in \overline{f^{-1}(\beta)}$, which proves $\bar{\beta}$ pulls back to the positive cone of a total order $\overline{f^{-1}(\beta)}$ on $A$.

The following example illustrates the above technique with $A=R$ and $f=\text{id}_{R}$:

\begin{ex}
Let $\alpha \subseteq \beta \subset R$, where $\overline{\alpha} \subset R/\supp(\alpha)$ totally orders the integral domain $R/\supp(\alpha)$, and $\overline{\beta} \subset R/\supp(\beta)$ totally orders the integral domain $R/\supp(\beta)$. We have a diagram
\begin{center}
\begin{tikzcd}
R \ar[d,swap,"\pi_{\alpha}"] \ar[rd,"\pi_{\beta}"] & \\ R/\supp(\alpha) & R/\supp(\beta),
\end{tikzcd}
\end{center}
and since $\supp(\alpha) \subseteq \ker(\pi_{\beta})$,
we may pass $\pi_{\beta}$ to the quotient to obtain a commutative diagram
\begin{center}
\begin{tikzcd}
R \ar[d,swap,"\pi_{\alpha}"] \ar[rd,"\pi_{\beta}"] & \\ R/\supp(\alpha) \ar[r,dotted,swap,"\overline{\pi_{\beta}}"] & R/\supp(\beta).
\end{tikzcd}
\end{center}

We may now totally order $R/\supp(\alpha)$ using either $\overline{\alpha}$ or $\overline{\pi_{\beta}}^{-1}(\overline{\beta})$. On the one hand, if $\overline{r} \in \overline{\alpha}$, fix an equivalence class representative $r$ of $\overline{r}$. Since $\alpha \subseteq \beta$, \[ \pi_{\beta}(r) \in \overline{\beta} \subset R/\supp(\beta). \]
We have now shown \[  \overline{\alpha} \subseteq \overline{\pi_{\beta}}^{-1}(\overline{\beta}).  \]
If there is an $x$ such that $x \in \overline{\pi_{\beta}}^{-1}(\overline{\beta})$ and $x \notin \overline{\alpha}$, since $\overline{\alpha}$ totally orders $R/\supp(\alpha)$, we must have $-x \in \overline{\alpha}$. We know $\overline{\alpha} \subseteq \overline{\pi_{\beta}}^{-1}(\overline{\beta})$, so $\{x,-x\} \subset \overline{\pi_{\beta}}^{-1}(\overline{\beta})$, forcing $x=0$. As $0$ \textit{is} an element of $\overline{\alpha}$ by definition, we have reached a contradiction and
\[ \overline{\alpha} = \overline{\pi_{\beta}}^{-1}(\overline{\beta}).  \]
\end{ex}

\subsection{The Topologies of Sper(-)}

We can now define the points of $\sper(A)$ for any commutative unital ring $A$. 

\begin{defn}
(\cite{rcs},2) Let $A$ be a commutative unital ring. The elements (read: points) of the real spectrum $\text{Sper}(A)$  are subsets $\alpha \subset A$ such that
\begin{enumerate}
    \item $\supp(\alpha)=\alpha \cap -\alpha$ is a prime ideal of $A$, and
    \item $\alpha/\supp(\alpha)$ is the positive cone of a total order on $A/\supp(\alpha)$.
\end{enumerate}
\end{defn}

We note that there is a map $\supp:\sper(A) \to \text{Spec}(A)$ given by $\alpha \mapsto \alpha \cap -\alpha$. I like to think of the image of $\supp$ as 'the set of prime ideals which are boundaries of positive cones of total order'. As we shall see in dealing with semialgebraic sets over a real closed field, the prime ideals correspond to varieties, and positive cones of total order correspond to regions these real varieties bound.

For the duration of this paper, wherever I specify a ring, I intend it to be commutative and unital. I may be less precise in my algebra moving forward, sometimes using the notation $a \in A/\supp(\alpha)$, for some $\alpha \in \sper(A)$, to mean $\pi_{\alpha}(a) \in A/\supp(\alpha)$. Fix a ring $A$. To define a topology on $\sper(A)$, basic open sets in $\sper(A)$ should be given by the nonvanishing of some 'function'. We will define the weak topology on $\sper(A)$ which is analogous to the Zariski topology, and the proconstructible topology which is analogous to the smallest topology on $\sper(A)$ containing the constructible subsets of $\sper(A)$ in the weak topology as closed sets. We explain the weak topology first, as it is required to define the proconstructible topology.

Let $a \in A$ and $\alpha \in \sper(A)$. We wish to evaluate 'the function' $a$ at the point $\alpha \in \sper(A)$. We have a ring homomorphism
\begin{center}
\begin{tikzcd}
A \ar[d,"\pi_{\alpha}"] & & \\ A/\supp(\alpha) \ar[r,hook] & \left(A/\supp(\alpha)\right)_{\langle 0 \rangle}, & 
\end{tikzcd}
\end{center}
and we fix the notation \[\kappa(\alpha):=\left(A/\supp(\alpha)\right)_{\langle 0 \rangle}. \]
I claim $\kappa(\alpha)$ is a totally ordered field, and I will show this by extending the total ordering on $A/\supp(\alpha)$. Fix distinct $a/b,c/d \in \kappa(\alpha)$, and define the following total order on $\kappa(\alpha)$;
\[ \frac{c}{d} \preceq \frac{a}{b} \iff ad-bc \in \overline{\alpha}=\pi_{\alpha}(\alpha).  \]
Since either $ad-bc \in \overline{\alpha}$ or  $bc-ad \in \overline{\alpha}$, and both are in $\overline{\alpha}$ precisely when 
\begin{align*}
& ad-bc=0 \\ \iff & ad=bc \; \text{ in } A/\supp(\alpha) \\ \iff & \frac{a}{b}=\frac{c}{d} \; \text{ in } \kappa(\alpha).
\end{align*}
I leave it as an exercise to show $\preceq$ has transitivity, and that the properties of addition and multiplication by an element of the positive cone play well with this ordering.

We have now established $\kappa(\alpha)$ is a totally ordered field, so by The Artin-Schreier Theorem, there is a unique algebraic extension $\kappa(\alpha) \subseteq \rho(\alpha)$, contained in the algebraic closure of $\kappa(\alpha)$ such that
\begin{itemize}
\item $\rho(\alpha)$ is a totally ordered field, and the ordering of $\rho(\alpha)$ extends the ordering of $\kappa(\alpha)$,
\item for all $x \in \rho(\alpha)$ with $0\preceq x$, there is an element $y \in \rho(\alpha)$ such that $y^{2}=x$.
\item Every polynomial $f(t) \in \rho(\alpha)[t]$ of odd degree has a root in $\rho(\alpha)$.
\end{itemize}
We call ordered fields of this kind \textit{real closed fields}, and there are many equivalent definitions of them. We call the final property the \textit{intermediate value property}. As examples, $\R$ and $\R_{alg}:=\R \cap \overline{\Q}$ are both real closed fiels. Please see pages 7-17 in \cite{rag} for further fundamental results on real closed fields.

We may now extend our previous diagram;
\begin{center}
\begin{tikzcd}
A \ar[d,swap,"\pi_{\alpha}"] \ar[rrd,dotted,"\rho_\alpha"] & & \\ A/\supp(\alpha) \ar[r,hook] & \kappa(\alpha) \ar[r,hook]  & \rho(\alpha).
\end{tikzcd}
\end{center}
We think of $\rho_\alpha(a)$ as the evaluation of the function $a$ at the point $\alpha$, and denote it $a(\alpha)$.

Let $R$ be a subring of $\rho(\alpha)$. Since $\rho(\alpha)$ is a ring containing $R$ with the intermediate value property, by Zorn's Lemma, there is a smallest ring $C(R)$ such that $R \subseteq C(R) \subseteq \rho(\alpha)$ and $C(R)$ has the intermediate value property.

\begin{defn}
If $R$ is a subring of a real closed field $\rho$, the \textit{convex hull} $C(R)$ of $R$ in $\rho$ is the smallest subring (by containment) of $\rho$ containing $R$ in which the intermediate value property holds. Generally, we may describe any totally ordered ring $C$ in which the intermediate value property holds as \textit{convex}. 
\end{defn}

Convex ideals are especially well-behaved in convex rings;
\begin{prop}
 Let $C$ be a convex ring with positive cone $C_+$, and let $I$ be a nonzero convex ideal of $C$. Then there exists a ring element $b \in C$ such that 
 \[ -b \leq x \leq b  \]
 for all $x \in I$.
\end{prop}
\begin{proof}
Firstly, note that if $b \in C$ bounds $I$ from above, then $-b$ bounds it from below, so it is sufficient to show $I$ has an upper bound. Towards a contradiction, suppose $I$ is unbounded. Let $a \in C_+$, and since $I$ is unbounded, choose $b \in I$ such that $b \geq a$. Since 
\[ a \in C_+, \, b-a \in C_+, \, \text{ and } b \in I \]
by convexity of $I$, $a \in I$. We chose $a$ arbitrarily, and so 
\[ 1 \in C_+ \subseteq I, \]
a contradiction.
\end{proof}

Perhaps more impressively, the convex prime ideals form a chain in any convex ring;
\begin{prop}
Let $C$ be a convex ring with positive cone $C_+$, and suppose $p_0$ and $p_1$ are nonzero convex prime ideals. Then either $p_0 \subseteq p_1$ or $p_1 \subseteq p_0$.
\end{prop}
\begin{proof}
 We first use a simple algorithm to determine which prime contains the other. The proof, following the algorithm, is straightforward.
\begin{enumerate}[i.]
\item Without loss of generality, fix $x \in p_0$, 
\item If $x \geq y$ for all $y \in p_{1}$, return $x$.
\item If not, there is an element $z_1 \in p_1$ such that $z_1 > x$.
\item If $z_1 \geq y$ for all $y \in  p_0$, return $z_1$.
\item If not, there is a $z_2 \in p_0$ such that $z_2>z_1$
\item Update $z_2 \mapsto x$ and proceed from ii.
\end{enumerate}
The algorithm above produces greater and greater elements of $p_1 \cup p_2$, and since $p_i$ is bounded for all $i$ by the previous proposition, so is $p_0 \cup p_1$. Since the order topology on $C$ is Hausdorff and $C$ is convex, the $z_i$'s have a limit and the algorithm terminates.
Let $x_{\infty} \in C$ denote the element on which the algorithm terminates. 
Since $x_{\infty} \in p_0 \cup p_1$, without loss of generality suppose $x_\infty \in p_1$. I claim that in this case, $p_0 \subseteq p_1$. If $a \in p_0$, one of $a$ and $-a$ are in $C_+$, so we assume $a \geq 0$ for ease of notation. By convexity, since
\[  a \in C_+, \, x_\infty-a \in C_+, \text{ and } x_\infty \in p_1 , \]
we know $a \in p_1$, proving this case.
If $x_{\infty} \notin p_0 \cup p_1$, $x_{\infty}$ must be the least upper bound for $p_0$ and $p_1$, implying that 
\[  p_0=(-x_{\infty},x_{\infty})=p_1,  \]
concluding the second and final case in this proof.
\end{proof}

Let 
\[ D(a)=\{\alpha \in \sper(A): a(\alpha)>0\}. \] 
If $a_1,a_2 \in A$, then
\[ D(a_1) \cap D(a_2)=\{\alpha \in \sper(A): a_{1}(\alpha)>0 \text{ and } a_{2}(\alpha)>0\}, \]
which is defined to be $D(a_{1},a_{2})$. Similarly, for any collection of elements $\{a_\lambda\}_{\lambda \in \Lambda} \subset A$, we have
\begin{align*}
 \bigcup_{\lambda \in \Lambda}D(a_\lambda)  &=\{\alpha \in \sper(A): \exists \lambda \in \Lambda \; s.t. \; a_{\lambda}(\alpha)>0\} , 
\end{align*}
but there may not exist an element $\hat{a}\in A$ such that 
\[ D(\hat{a})=\{\alpha \in \sper(A): \exists \lambda \in \Lambda \; s.t. \; a_{\lambda}(\alpha)>0\}. \]

\begin{ex}
Let $\rho$ be the real closure of the ordered field $\Q$ in $\overline{\Q}$, and let $A=\rho[t]$. Let
\[ \{a_\lambda\}_{\lambda \in \Lambda}=\left\{t-\frac{1}{n} \right\}_{n \in \N}. \]
In $\sper(\rho[t])$, we then compute
\begin{align*}
\bigcup_{n \in \N}D\left(t-\frac{1}{n} \right) &= \{\alpha \in \sper(\rho[t]):\exists n \in \N \; s.t. \; a_{n}(\alpha)>0\}.
\end{align*}
We observe 
\[ D\left(t-\frac{1}{n} \right)=\left\{ \alpha \in \sper(\rho[t]):\overline{t-\frac{1}{n}}\in \overline{\alpha} \subseteq \rho[t]/\supp(\alpha) \right\}, \]
as $\rho[t]/\supp(\alpha)$ embeds in its real closure $\rho(\alpha)$ as an ordered ring. As $1/n$ is positive for any natural $n$, we know
\[ \left\{ \alpha \in \sper(\rho[t]):\overline{t-\frac{1}{n}}\in \overline{\alpha} \subseteq \rho[t]/\supp(\alpha) \right\} \] is contained in \[ \left\{ \alpha \in \sper(\rho[t]):\overline{t-\frac{1}{n}}+\overline{\frac{1}{n}}\in \overline{\alpha} \subseteq \rho[t]/\supp(\alpha) \right\},  \]
or more succinctly
\[ D\left(t-\frac{1}{n} \right) \subseteq D(t) \]
for all $n \in \N$. Hence the union over all natural $n$ is also contained in $D(t)$;
\[ \bigcup_{n \in \N} D\left(t-\frac{1}{n} \right) \subseteq D(t). \]
Let \[ \alpha_{0}=\{f(t) \in \rho[t]: f(0) \geq 0\}. \]
We then have $\alpha_{0} \in D(t)$ by definition, but 
\[ \overline{\left(t-\frac{1}{n}\right)}(\alpha_{0})=\overline{\left(-\frac{1}{n}\right)}(\alpha_{0})<0 \]
in $\rho[t]/\langle t \rangle$ for all natural $n$. Hence 
\[ \bigcup_{n \in \N} D\left(t-\frac{1}{n} \right) \subset D(t). \]
\end{ex}
Observing that $D(1)=\sper(A)$ and $D(-1)=\emptyset$, these subsets of $D(a_1,\hdots,a_n) \subseteq \sper(A)$ form a basis for a topology on $\sper(A)$;
\begin{defn}
(\cite{rcs},4) The set $\sper(A)$, together with the topology generated by subsets of $\sper(A)$ of the form $D(a_1,\hdots,a_n)$, is called the \textit{real spectrum} of $A$, and this topology is called the \textit{weak topology}.
\end{defn}
I will sometimes write 
\[ P(a):=\{\alpha \in \sper(A): a(\alpha) \geq 0\}=D(-a)^{c}. \]
The \textit{constructible} subsets of the real spectrum are those which can be written using using basis open sets and finitely many applications of union, intersection, and complement. For example
\[  P(a) \cap D(a)=D(a) \]
is constructible in $\sper(A)$.
\begin{defn}
The smallest topology on $\sper(A)$ containing the constructible subsets of $\sper(A)$ in the weak topology as \underline{closed} sets is called the \textit{proconstructible topology} on $\sper(A)$.
\end{defn}

If I do not specify a topology, I always mean the weak topology.
If $f:R \to A$ is a ring homomorphism, we have shown previously that for subsets $\alpha \subset A$ which project to positive cones of total orders under $A \to A/\supp(\alpha)$, we have a commutative diagram
\begin{center}
\begin{tikzcd}
R \ar[r,"f"] \ar[d] & A \ar[d] \\ R/\supp(f^{-1}(\alpha)) \ar[r,dotted,"\exists \bar{f}"] & A/\supp(\alpha).
\end{tikzcd}
\end{center}
In the language of topology, this means we have a map of sets
\[ \phi:\sper(A) \to \sper(R), \]
\begin{prop}
If $f:R \to A$ is a ring homomorphism, and $\phi:\sper(A) \to \sper(R)$ is the corresponding map of sets, then $\phi$ is continuous in the weak topology.
\end{prop}
\begin{proof}
 Let $r$ be a nonzero element of $R$, and consider the diagram 
\begin{center}
\begin{tikzcd}
\sper(A) \ar[r,"\phi"] & \sper(R) \\  & D(r), \ar[u,hook,"\circ" description] 
\end{tikzcd}
\end{center}
which we may extend to a diagram
\begin{center}
\begin{tikzcd}
\sper(A) \ar[r,"\phi"] & \sper(R) \\ \phi^{-1}(D(r)) \ar[u,hook]  \ar[r, "\phi|_{\phi^{-1}(D(r))}"]  & D(r). \ar[u,hook,"\circ" description] 
\end{tikzcd}
\end{center}
We have 
\begin{align*}
\phi^{-1}(D(r))=&\{\alpha \in \sper(A):r(\phi(\alpha))>0\} \\=&\{\alpha \in \sper(A):r(f^{-1}(\alpha))>0\},
\end{align*}
and we will proceed to show this set is a basis open set in $\sper(A)$. Now, $r(f^{-1}(\alpha))>0$ holds precisely when $r \in f^{-1}(\alpha)/\supp(f^{-1}(\alpha))$ and $ r \neq 0$. We then know 
\[ f(r) \in \alpha/\supp(\alpha), \]
and so
\begin{align*} 
\{\alpha \in \sper(A):r(f^{-1}(\alpha))>0\}&=\{ \alpha \in \sper(A):f(r)(\alpha)>0 \} \\&= D(f(r)),
\end{align*}
proving not only that $\phi^{-1}(D(r))$ is open, but also that it equals $D(f(r))$!
\end{proof}

\begin{defn}
(\cite{rcs},4) For rings $A$ and $B$, a map $\varphi:\sper(A) \to \sper(B)$ is a \textit{morphism of spectral spaces} if it is continuous in the weak and proconstructible topologies.
\end{defn}
As you might expect, if $f:B \to A$ is a ring homomorphism, the corresponding map $\varphi:\sper(A) \to \sper(B)$ is a morphism of spectral spaces.

Now, we might wonder about fibers over points. For starters, if $\rho_{\alpha}:A \to \rho(\alpha)$ is the evaluation morphism defined previously, \[\rho_{\alpha}^{-1}(\{x \in \rho(\alpha):x \geq 0\})=\alpha, \] so this yields a bijective correspondence between points $\alpha \in \sper(A)$ and spaces $\sper(\rho(\alpha))$. Moving forward, we identify $\alpha$ (with the subspace topology) and $\sper(\rho(\alpha))$.

It turns out the easiest way to see that \textit{all} fiber products of affine real spectra exist is by studying the real closure operation on a ring $A$ as a left-adjoint to a certain forgetful functor.


\subsection{A Model-Theoretic Moment}

We have a model-theoretic way of dealing with proconstructible subsets of $\sper A$ which is a bit nicer than working with the definition. Let $\mathscr{L}_{\leq}$ be the language of ordered fields, and let Th be the theory of real closed fields in $\mathscr{L}_{\leq}$. We may extend $\mathscr{L}_{\leq}$ by adding the constants  
\[ \{c_a:a \in A \} \]
to it to obtain $\mathscr{L}_{\leq}(A)$. We may then extend the functions '$+$' and '$*$' in $\mathscr{L}_{\leq}(A)$ so that for all $a,b \in A$,
\[  c_0=0, \; c_1=1, \; c_a+c_b=c_{a+b}, \text{ and } c_a*c_b=c_{ab}.  \]
A model of $\text{Th}(A)$ in $\mathscr{L}_{\leq}(A)$ is a real closed field $F$, such that for each $a,b \in A$, there exist constants $c_a , c_b \in F$ such that 
\[  c_0=0, \; c_1=1, \; c_a+c_b=c_{a+b}, \text{ and } c_a*c_b=c_{ab}.  \] 
Consequently, we may also think of a model of $\text{Th}(A)$ in $\mathscr{L}_{\leq}(A)$ as a homomorphism $\rho:A \to F$ to a real closed field $F$.

If $F_{+}$ denotes the positive cone of total order for $F$, by Theorem 5, $\rho^{-1}(F_+)$ is the positive cone of a partial order on $A$, and it descends to a total order on $A/\ker(\rho)$. Hence $\rho^{-1}(F_+)$ is a point in $\sper A$, and conversely, each point $\alpha \in \sper A$ gives a homomorphism to a real closed field $\rho_{\alpha}:A \to \rho(\alpha)$. So the points of $\sper A$ are in bijection with the collection of homomorphisms $\rho:A \to F $ modulo the relation
\[ (\rho_1:A \to E) \sim (\rho_2:A \to F) \iff \rho_{1}^{-1}(E_+)=\rho_{2}^{-1}(F_+). \]

We can also describe constructible subsets of $\sper A$ using model theory. Fix $a \in A$, and consider the open set 
\[ D(a) \subset \sper A. \]
We know 
\[  D(a)=\{ \alpha \in \sper A: a(\alpha)>0 \},  \]
and $a(\alpha)=\rho_{\alpha}(a)$, where $\rho_{\alpha}:A \to \rho(\alpha)$ is the evaluation map. For each $\alpha \in D(a)$, consider the equivalence class of homomorphisms $(\rho_{\alpha}:A \to \rho(\alpha))$ that satisfy the sentence
\[ \varphi: \; c_a >0  \]
parameterized by $a$ in $\mathscr{L}_{\leq}(A)$. Under the previous correspondence,
\[ D(a)=\{(\rho:A \to F): \rho \models \varphi \}, \]
and more generally, since constructible subsets are finite boolean combinations of sets of the above form, they can also be described by sentences $\varphi(a_1,\hdots,a_m)$ with finitely many parameters $a_1,\hdots,a_m$ from $A$. Intersections of constructible sets correspond to combining sentences with $\wedge$, so infinite intersections \footnote{These are the interesting pro-constructible subsets!} correspond to 'infinite-length' sentences in $\mathscr{L}_{\leq}(A)$.

If $B \to A$ is a ring homomorphism, (pro)constructible subsets $L \subseteq \sper B$ have (pro)constructible inverse images $K \subseteq \sper A$ under the induced morphism of spectral spaces. However, (pro)constructible subsets of $\sper A$ are not necessarily sent to (pro)constructible subsets of $\sper B$ under the induced map of spectral spaces. In some fairly general cases, however, we get lucky and such subsets are preserved as described below.

Let $f:R \to A$ be a  homomorphism with induced map $\pi:\sper A \to \sper R$ such that $A$ is a finitely presented algebra over $R$;
\[ R \to A \cong R[x_1,\hdots,x_n]/\langle Q_1,\hdots,Q_k \rangle. \]
Let $\varphi(a_1,\hdots,a_m)$ be a sentence in $\mathscr{L}_{\leq}(A)$ with parameters $a_i$ describing a constructible subset $K$ of $\sper A$. Each polynomial $a_i$ has coefficients $a_{i,1},\hdots,a_{i,\deg(a_i)}$ in $R$, and similarly for the $Q_i$. Let $\Phi$ be the sentence over $\mathscr{L}_{\leq}(R)$, with the coefficients of the $a_i$ and $Q_j$ as parameters, and in the variables $X_{i}$ for $i \in \{1, \hdots,n\}$ which says
\[ \exists X_1,\hdots, X_n : \varphi(a_1(X_1,\hdots,X_n),\hdots,a_m(X_1,\hdots,X_n)) \wedge \left(\bigwedge_{j=1}^{m}Q_{j}(X_1,\hdots,X_n)=0 \right). \]
I claim this sentence exactly describes $\pi(K)$. If $(R\to \rho(\beta))$ is in $\pi(K)$, there exists a point $(A \to \rho(\alpha))$ in $K$ such that 
\begin{center}
\begin{tikzcd}
\sper A \ar[r,"\pi"] & \sper R \\ \sper \rho(\alpha) \ar[u,hook,"\iota_{\alpha}"] \ar[r] & \sper \rho(\beta) \ar[u,hook,"\iota_{\beta}" swap]
\end{tikzcd}
\end{center}
commutes. Consequently, we have a commutative diagram of ring homomorphisms
\begin{center}
\begin{tikzcd}
R \ar[r,"f"] \ar[d,"\rho_{\beta}" swap] & A \ar[d,"\rho_{\alpha}"] \\ \rho(\beta) \ar[r] & \rho(\alpha).
\end{tikzcd}
\end{center}
Since $R\to \rho(\alpha)$ factors through $A \to \rho(\alpha)$ and $(A \to \rho(\alpha))$ satisfies $\varphi$, $(R \to \rho(\alpha))$ satisfies $\Phi$, proving that the elements of $\pi(K)$ satisfy $\Phi$. Suppose $(R\to \rho(\beta))$ satisfies $\Phi$. Using elementary ring theory, we can construct a homomorphism $\lambda:A \to \rho(\beta)$ such that 
\begin{center}
\begin{tikzcd}
R \ar[r,"f"] \ar[d,"\rho_{\alpha}" swap] & A \ar[ld,dotted,"\lambda" ] \\ \rho(\beta) &
\end{tikzcd}
\end{center}
commutes, where $\lambda(x_i)=X_i$ in $\rho(\beta)$. Hence
\begin{center}
\begin{tikzcd}
\sper A \ar[r,"\pi"]  & \sper R  \\ \sper \rho(\beta) \ar[u,"\iota_{\alpha}" swap] \ar[ru,dotted,"l" ] &
\end{tikzcd}
\end{center}
commutes, and $(R \to \rho(\beta))$ is in $\pi(K)$. We have now proved
\begin{thm}
(\cite{coste_roy_Topologie_1982},34) Let $f:R \to A$ be a ring homomorphism giving a finite presentation of $A$ over $R$. Let $\pi:\sper A \to \sper R$ be the induced morphism of spectral spaces, and let $K$ be a (pro)constructible subset of $\sper A$. Then $\pi(K)$ is (pro)constructible.
\end{thm}

\subsection{A Sheaf on Sper(-)}

We are ready to define the structure sheaf of the real spectrum! 
Let $A$ be a ring with real spectrum $X:=\sper(A)$ in its weak topology. The structure homomorphism $p:A \to A[T]$ yields a morphism of spectral spaces $\pi:\sper A[T] \to \sper A$. The set of morphisms of spectral spaces \[ \{\sigma:\sper A \to \sper A[T]: \pi \circ \sigma=\text{id}_{\sper A}\} \]
corresponds bijectively to the evaluation maps 
\[  e_{a}:A[T] \to A \; \text{ defined by } T \mapsto a.  \]
These morphisms $e_{a}$ have a natural ring structure; for $a,b \in A$, we have
\begin{align*}
e_a +e_b =& e_{a+b} \quad \text{, and } \\  e_a e_b =& e_{ab}.
\end{align*}
Hence sections of $\pi:\sper A[T] \to \sper A$ have a ring structure which contains a ring isomorphic to $A$.

Let $K \subseteq \sper A$ be a proconstructible subset of $\sper A$. $\pi^{-1}(K)$ is the inverse image of an intersection of constructible sets $\bigcap_{\gamma \in \Gamma}C_{\gamma}$, and since constructible sets are preserved under inverse image,
\[ \pi^{-1}(K)=\pi^{-1}\left( \bigcap_{\gamma \in \Gamma}C_{\gamma} \right)=\bigcap_{\gamma \in \Gamma}\pi^{-1}(C_{\gamma})  \]
is proconstructible. Consider the 'local sections' \[ \left\{\sigma:K \to \pi^{-1}(K) \biggr\lvert \sigma(K) \text{ is closed constructible, and } (\pi\lvert_{\pi^{-1}(K)}) \circ \sigma=\text{id}_{K}\right\}. \]
In the paragraphs to come, we will see these sections carry a ring structure just like the global sections \cite{rcs}! 

\begin{defn}
Let $X=\sper A$. Define
\[ \mathscr{O}_{X}(X):=\{\sigma:X \to \sper A[T]: \pi \circ \sigma=\text{id}_{X}\} \]
as the \textit{global sections of the structure sheaf of} $\sper A$, and define 
\[ \mathscr{O}_{X}(K):= \left\{\sigma:K \to \pi^{-1}(K) \biggr\lvert \sigma(K) \text{ is closed constructible, and } \pi\lvert_{\pi^{-1}(K)} \circ \sigma=\text{id}_{K}\right\}  \]
as the \textit{local sections of the structure sheaf } $\mathscr{O}_X$.
\end{defn}

If $\alpha$ is a point in  $K$, a proconstructible subset of $\sper A$, by  Proposition 4.3 of \cite{coste_roy_Topologie_1982},  we have a diagram 
\begin{center}
\begin{tikzcd}
& \sper A[t] \ar[r,"\pi"]  & \sper A \\ \pi^{-1}(\alpha) \ar[r,"\cong" ] & \sper \rho(\alpha)[t] \ar[r,"\pi|"] \ar[u,hook] & \sper \rho(\alpha). \ar[u,hook]
\end{tikzcd}
\end{center}
If $\sigma \in \mathscr{O}_{\sper A}(K)$, then
\[ \{\sigma(\alpha)\}_{\alpha \in K}=\sigma(K) \cap \pi^{-1}(\alpha) \]
is closed constructible in $\sper \rho(\alpha)[t]$. If $r \in \rho(\alpha)$, we have a corresponding subspace 
\[  \{ \gamma \in \sper \rho(\alpha)[t]: t(\gamma)=r  \}, \]
and since we can recover $r$ by evaluating $t$ at any point in the set, the correspondence is bijective. Under this correspondence, $\sigma(\alpha) \in \rho(\alpha)$, and
\[  \sigma(K) \subseteq \prod_{\alpha \in K}\rho(\alpha).  \]
Moving forward, we will also denote the map
\[ K \to \prod_{\alpha \in K}\rho(\alpha)  \]
by $\sigma$. Letting $X=\sper A$, we obtain a map 
\[ i:\mathscr{O}_X(K) \to \prod_{\alpha \in K}\rho(\alpha) \]
which takes $\sigma$ and maps it to \[(\sigma(\alpha))_{\alpha \in K}. \]
Since different sections $\sigma_1,\sigma_2:K \to \pi^{-1}(K)$ must take different values at some point $\gamma \in K$, $i$ is an injective \textbf{set map} of rings. We aim to show $\mathscr{O}_X(K)$ is a subring of $\prod_{\alpha \in K}\rho(\alpha)$ \cite{rcs}.

Before proceeding, we need to know that if $\alpha,\beta \in K$ are points such that $\alpha \subseteq \beta$, we have $\sigma(\alpha)=\sigma(\beta)$ (in some sensible field). In the paragraphs below, we develop the algebraic machinery to guarantee exactly this.

\begin{defn}
(\cite{rcs},6) Fix a ring $A$, and let $\alpha,\beta \in \sper(A)$. We say $\alpha$ \textit{specializes to} $\beta$, or $\beta$ \textit{is a specialization of} $\alpha$ if $\beta$ is in the closure of $\alpha$ in the weak topology, or $\alpha \subseteq \beta$. We similarly say $\alpha$ is a \textit{generization} of $\beta$.
\end{defn}

\begin{lemma}
(\cite{rcs},6) Let $\alpha, \beta \in \sper A$, with $\beta$ a specialization of $\alpha$. Then $\,\textbf{} \supp(\beta)/\supp(\alpha)$ is a convex ideal in $A/\supp(\alpha)$.
\end{lemma}
\begin{proof}
Let $a$ and $b$ be elements of $A/\supp(\alpha)$, and suppose $a$ and $b-a$ are in $\alpha/\supp(\alpha)$, and $b \in I=\supp(\beta)/\supp(\alpha)$. We want to show $a \in I$. We have a projection
\[ p:A/\supp(\alpha) \to A/\supp(\beta), \]
and we use this projection to analyze the cases $a \neq 0$ and $a=0$ separately. If $a=0$, $a \in I$ because $I$ is an ideal. If $a \neq 0$,
\[ p(\{a,b-a\})=\{p(a),p(b-a)\}=\{p(a),-p(a)\}\subset \beta/\supp(\beta),  \]
and since $\beta/\supp(\beta)$ totally orders $A/\supp(\beta)$, $p(a)=0$. Hence $a \in I$.
\end{proof}

Now, given $\alpha$ and $\beta$ as in the above lemma, we have a diagram
\begin{center}
\begin{tikzcd}
 A \ar[r] \ar[rd] & A/\supp(\alpha) \ar[r,"\rho_\alpha"] \ar[d] & \rho(\alpha) \\ & A/\supp(\beta) \ar[r,"\rho_\beta"] & \rho(\beta), 
\end{tikzcd}
\end{center}
which commutes since $\alpha \subseteq \beta$. By Zorn, there is a largest convex subring $C_{\beta \alpha} \subseteq \rho(\alpha)$ with maximal ideal $m_{\beta \alpha}$ such that 
\begin{center}
\begin{tikzcd}
 & C_{\beta \alpha} \ar[d,hook] \\ A/\supp(\alpha) \ar[d] \ar[r,"\rho_\alpha"] \ar[ru,hook,"f"] & \rho(\alpha)  \\  A/\supp(\beta) \ar[r,"\rho_\beta"] & \rho(\beta)
\end{tikzcd}
\end{center}
commutes, and \[ f^{-1}(m_{\beta \alpha})=\supp(\beta)/\supp(\alpha). \]

I claim $\rho_{\beta \alpha}=C_{\beta \alpha}/m_{\beta \alpha}$ is a real closed field, and by \cite{prestel}, it suffices to show $\rho_{\beta \alpha}$ is uniquely and maximally ordered, in such a  way that with this ordering, $\rho_{\beta \alpha}$ is an ordered field. Since $C_{\beta \alpha}$ is convex in $\rho(\alpha)$, the positive cone $C_+ \subset C_{\beta \alpha}$ is simply
\[ C_+=C_{\beta \alpha} \cap \{y \in \rho(\alpha): \exists x \in \rho(\alpha) \text{ such that } x^{2}=y\}. \]
In the quotient $C_{\beta \alpha} \to \rho_{\beta \alpha}$, $C_+/m_{\beta \alpha}$ remains closed under addition and multiplication, so it is left to check that at most one of $x$ and $-x$ are in $C_+/m_{\beta \alpha}$ for all $x \in $. Without loss of generality, assume $-x \notin C_+/m_{\beta \alpha}$. We then have 
\begin{align*}
&-x \notin C_+/m_{\beta \alpha} \\ \implies & -x+a \notin C_+ \text{ for all } a \in m_{\beta \alpha} \\ \implies & -x \notin C_+ \\ \implies & x \in C_+ \\ \implies & x \in C_+/m_{\beta \alpha},
\end{align*}
which implies $C_+/m_{\beta \alpha}$ is the positive cone of a total order on $\rho_{\beta \alpha}$. Since $C_{\beta \alpha} \subseteq \rho(\alpha)$, and $\rho(\alpha)$ is uniquely and maximally ordered, $C_{\beta \alpha}$ is uniquely and maximally ordered. Hence the quotient $\rho_{\beta \alpha}$ is maximally ordered with positive cone equal to its squares, so it is real closed by \cite{prestel}.

We may now extend our diagram to
\begin{center}
\begin{tikzcd}
 & C_{\beta \alpha} \ar[d,hook] \ar[rd,"\pi_{\beta \alpha}"] & \\ A/\supp(\alpha)  \ar[d] \ar[r,"\rho_\alpha"] \ar[ru,hook,"f"]  & \rho(\alpha) & \rho_{\beta \alpha} \\  A/\supp(\beta) \ar[r,"\rho_\beta"] & \rho(\beta), &
\end{tikzcd}
\end{center}
and since $\supp(\beta)$ is in the kernel of \[ A/\supp(\alpha) \to C_{\beta\alpha} \to \rho_{\beta \alpha}, \] passing to the quotient by $\supp(\beta)$ yields a commutative diagram
\begin{center}
\begin{tikzcd}
 & C_{\beta \alpha} \ar[d,hook] \ar[rd,"\pi_{\beta \alpha}"] & \\  A/\supp(\alpha) \ar[d] \ar[r,"\rho_\alpha"] \ar[ru,hook,"f"]  & \rho(\alpha)  & \rho_{\beta \alpha} \\  A/\supp(\beta) \ar[r,"\rho_\beta"] \ar[rru,crossing over,dotted] & \rho(\beta). &
\end{tikzcd}
\end{center}
By the universal properties of the fraction field and of real closed fields, $A/\supp(\beta) \to \rho_{\beta \alpha}$ factors through $\rho(\beta)$ to yield the final commutative diagram
\begin{center}
\begin{tikzcd}
 & C_{\beta \alpha} \ar[d,hook] \ar[rd,"\pi_{\beta \alpha}"] & \\  A/\supp(\alpha) \ar[d] \ar[r,"\rho_\alpha"] \ar[ru,hook,"f"]  & \rho(\alpha)  & \rho_{\beta \alpha} \\  A/\supp(\beta) \ar[r,"\rho_\beta"]  & \rho(\beta) \ar[ru,dotted,"\iota" swap]. &
\end{tikzcd}
\end{center}

\begin{remark}
Suppose $A=\Q[x,y]$, and define $\alpha \in \sper A$ as
\[  \alpha=\left\{f \in A: \exists \text{ rational } \epsilon >0 \text{ such that } f\biggr\lvert_{[0,\epsilon] \times [0,\epsilon]} \geq 0 \right\},  \]
and $\beta \in \sper A$ as
\[ \beta=\{ f \in A: f(0,0) \geq 0 \}.  \]
We have by definition that $\beta$ is a specialization of $\alpha$, i.e. that $\alpha \subset \beta$. Using our model theoretic definition of $\sper A$, $\alpha$ corresponds to the homomorphism $\Q[x,y] \to \R_{\text{alg}}(\epsilon)$ with $x,y \mapsto \epsilon$, and $\R_{\text{alg}}(\epsilon)$ is the smallest real closed field containing $\R_{\text{alg}}$ and some $\epsilon$ greater than zero and smaller than every other positive real algebraic number. The point $\beta \in \sper A$ corresponds to the homomorphism $\Q[x,y] \to \R_{\text{alg}}$ given by evaluation at $(0,0)$. Note that \textbf{there are no field homomorphisms } $\rho(\alpha) \to \rho(\beta)$! Even though there is a real closed field homomorphism \[\rho(\beta)=\R_{\text{alg}} \hookrightarrow \R_{\text{alg}}(\epsilon)=\rho(\alpha), \]
this homomorphism does not come from the geometry or topology of $\sper A$. 
To see where this homomorphism is coming from, note that we can factor the homomorphisms
\begin{center}
\begin{tikzcd}
\Q[x,y] \ar[r] \ar[rd] & \R_{\text{alg}} \\  &  \R_{\text{alg}}(\epsilon)    
\end{tikzcd}
\end{center}
as 
\begin{center}
\begin{tikzcd}
\Q[x,y] \ar[r,hook] & \R_{\text{alg}}[x,y] \ar[r] \ar[rd] & \R_{\text{alg}} \\ & &  \R_{\text{alg}}(\epsilon).    
\end{tikzcd}
\end{center}
Since the surjective homomorphism $\R_{\text{alg}}[x,y] \to \R_{\text{alg}}$ has a splitting, we get the homomorphism 
\begin{center}
\begin{tikzcd}
\R_{\text{alg}}[x,y] \ar[rd] & \R_{\text{alg}} \ar[l,dotted] \ar[d,dotted] \\  &  \R_{\text{alg}}(\epsilon).    
\end{tikzcd}
\end{center}
When working with rings of finite type over ordered fields, there are many other non-geometric coincidences that occur, such as consequences of the Tarski-Seidenberg principle, or getting a morphism from the universal property of real closed fields.  
\end{remark}

If points can live within other points in the real spectrum, the analogue of a continuous function should be reasonable with respect to the lattice structure of these points;
\begin{defn}
(\cite{rcs},7) Let $K$ be a (pro)constructible subset of $\sper A$. An element $\sigma$ of the ring $\prod_{\alpha \in K}\rho(\alpha)$ is a \textit{compatible section} if for all $\alpha$ and $\beta$, with $\beta$ a specialization of $\alpha$, we have
\begin{enumerate}
\item $\sigma(\alpha) \in  C_{\beta \alpha}$, and
\item $\pi_{\beta \alpha}(\sigma(\alpha))=\iota(\sigma(\beta))$ in $\rho_{\beta \alpha}$.
\end{enumerate}
\end{defn}

Let $\sigma_{id}:K \to \prod_{\alpha \in K}\rho(\alpha)$ send $\alpha \in K$ to $1 \in \rho(\alpha)$. Then, for each $\alpha,\beta \in K$ with $\beta \subseteq \overline{\alpha}$, we know 
\[ \sigma_{id}(\alpha)=1 \in C_{\beta \alpha}. \]
We also have 
\[ \pi_{\beta \alpha}(1)=\iota(1)=1, \]
so compatible sections exist, and $\sigma_{id}$ is a compatible section. 

In fact, we can show each element of $A$ gives a compatible section! Since 
\begin{center}
\begin{tikzcd}
& & C_{\beta \alpha} \ar[d,hook] \ar[rd,"\pi_{\beta \alpha}"] & \\ A \ar[r] \ar[rd] & A/\supp(\alpha)  \ar[r,"\rho_\alpha"] \ar[d] \ar[ru,hook,"f"]  & \rho(\alpha)  & \rho_{\beta \alpha} \\ & A/\supp(\beta) \ar[r,"\rho_\beta"]  & \rho(\beta) \ar[ru,dotted,"\iota" swap] &
\end{tikzcd}
\end{center}
commutes from $A$, the diagram commutes for every $\alpha,\beta \in K$ with $\beta \in \overline{\alpha}$, and elements of $\prod_{\alpha \in K}\rho(\alpha)$ which correspond to elements of $A$ via 
\[  A \to \prod_{\alpha \in K}\rho(\alpha) \]
are all compatible sections. 

Let $\sigma_1$ and $\sigma_2$ are compatible sections on a constructible subspace $K \subseteq \sper A$, and consider their sum $\sigma_1 + \sigma_2$. For each $\alpha,\beta \in K$ with $\beta \in \bar{\alpha}$, we have 
\[ (\sigma_1 + \sigma_2)(\alpha)=\sigma_1(\alpha) + \sigma_2(\alpha) \in C_{\beta \alpha}, \]
and
\begin{align*}
\pi_{\beta \alpha}((\sigma_1+\sigma_2)(\alpha))=&\pi_{\beta \alpha}(\sigma_1(\alpha))+\pi_{\beta \alpha}(\sigma_2(\alpha))\\=&\iota(\sigma_1(\beta))+\iota(\sigma_2(\beta))\\=&\iota((\sigma_1+\sigma_2)(\beta))
\end{align*}
Hence the sum of compatible sections is compatible, and continuing in a similar manner, one can prove
\begin{cor}
(\cite{rcs},7) The set of compatible sections $\sigma \in \prod_{\alpha \in K}\rho(\alpha)$ is a ring.
\end{cor}

\begin{thm}
(\cite{rcs},7) We consider $\prod_{\alpha \in K}\rho(\alpha)$ as a subset of $\pi^{-1}(K) \subseteq \sper A[t]$ as before, where $\pi:\sper A[t] \to \sper A$ is induced by the $A$-algebra structure homomorphism. An element $\sigma \in \prod_{\alpha \in K}\rho(\alpha)$ is a compatible section if and only if the following condition holds:
\begin{center}
If $\alpha,\beta \in K$ with $\beta \in \overline{\alpha}$, $\sigma(\beta)$ is the unique specialization of $\sigma(\alpha)$ in $\pi^{-1}(\beta)$.
\end{center}
\end{thm}
\begin{proof}
Before we prove the above, let's establish some crucial facts. Since $\sigma:K \to \pi^{-1}(K)$ is a local section of $\pi$,
\begin{center}
\begin{tikzcd}
K \ar[r,"\sigma"] & \pi^{-1}(K) \\ \sper \rho(\alpha) \ar[r,"\sigma"] \ar[u,hook] & \sper \rho(\sigma(\alpha)), \ar[u,hook]
\end{tikzcd}
\end{center}
commutes, and since $\pi(\sigma(\alpha))=\alpha$,
\begin{center}
\begin{tikzcd}
\sper \rho(\alpha) \ar[r,"\sigma"] \ar[rd,"id" swap]  & \sper \rho(\sigma(\alpha)) \ar[d,"\pi"] \\ & \sper \rho(\alpha)
\end{tikzcd}
\end{center}
also commutes. On the level of ring homomorphisms, we have
\begin{center}
\begin{tikzcd}
\rho(\alpha) \ar[r,hook] \ar[rr, bend left=20, "id"] & \rho(\sigma(\alpha)) \ar[r,hook] & \rho(\alpha),
\end{tikzcd}
\end{center}
proving $\rho(\alpha)=\rho(\sigma(\alpha))$. We then have a diagram of $A$-algebra homomorphisms
\begin{center}
\begin{tikzcd}
A \ar[rd,"\rho_{\alpha}"] \ar[d] & \\ A[t] \ar[r,"\rho_{\sigma(\alpha)}" swap] & \rho(\alpha),
\end{tikzcd}
\end{center} 
which commutes because 
\begin{center}
\begin{tikzcd}
\sper A[t] \ar[r,"\pi"] & \sper A  \\ \sper \rho(\alpha) \ar[u,hook] \ar[r,"="]  & \sper \rho(\alpha) \ar[u,hook]
\end{tikzcd}
\end{center}
does. The element $\rho_{\sigma(\alpha)}(t) \in \sper \rho(\alpha)$ must be in the image of $\rho_\alpha$, because the diagram still commutes if we include the homomorphism induced by the local section $\sigma$;
\begin{center}
\begin{tikzcd}
A \ar[rd,"\rho_{\alpha}"] \ar[d,bend left=20] & \\ A[t] \ar[r,"\rho_{\sigma(\alpha)}" swap] \ar[u,bend left=20] & \rho(\alpha).
\end{tikzcd}
\end{center}
Consequently, when $\sigma(\alpha)$ is considered an element of $\rho(\alpha)$, it equals $\rho_{\sigma(\alpha)}(t)$ by commutativity.

Now, let $\sigma \in \prod_{\alpha \in K}\rho(\alpha)$ be a compatible section, and recall the commutative diagram
\begin{center}
\begin{tikzcd}
& C_{\beta \alpha} \ar[d,hook] \ar[rd,"\pi_{\beta \alpha}"] & \\ A/\supp(\alpha) \ar[d] \ar[r,"\rho_\alpha"] \ar[ru,hook,"f"] & \rho(\alpha)  & \rho_{\beta \alpha} \\  A/\supp(\beta) \ar[r,"\rho_\beta"]  & \rho(\beta) \ar[ru,dotted,"\iota" swap], &
\end{tikzcd}
\end{center}
where $\alpha,\beta \in K$ and $\beta \in \bar{\alpha}$. We have a corresponding commutative diagram for $\sigma(\beta)$, which is contained in $\overline{\sigma(\alpha)}$;
\begin{center}
\begin{tikzcd}
& C_{\sigma(\beta) \sigma(\alpha)} \ar[d,hook] \ar[rd,"\pi_{\sigma(\beta) \sigma(\alpha})"] & \\ A[t]/\supp(\sigma(\alpha)) \ar[d] \ar[r,"\rho_{\sigma(\alpha)}"] \ar[ru,hook] & \rho(\alpha)  & \rho_{\sigma(\beta) \sigma(\alpha)} \\  A[t]/\supp(\sigma(\beta)) \ar[r,"\rho_{\sigma(\beta)}"]  & \rho(\beta) \ar[ru,dotted,"\iota" swap]. &
\end{tikzcd}
\end{center}
We then have
\begin{align*}
 &\iota \circ \rho_{\sigma(\beta)}(t) \\ =& \iota(\sigma(\beta)) \\ =& \pi_{\sigma(\beta) \sigma(\alpha)}(\sigma(\alpha)) \\ =& \pi_{\sigma(\beta) \sigma(\alpha)} \circ \rho_{\sigma(\alpha)}(t),
\end{align*}
and $\sigma(\beta)$ and $\sigma(\alpha)$ are equal as elements of $\rho_{\sigma(\beta) \sigma(\alpha)}$. It is left to show that this specialization is unique.

Suppose $\gamma \in \pi^{-1}(\beta)$ is such that $\gamma \in \overline{\sigma(\alpha)}$. Since \[ \{\sigma(\beta)\}=\sigma(K) \cap \pi^{-1}(\beta), \]
and $\sigma(K)$ is closed constructible in $\sper A[t]$, $\{\sigma(\beta)\}$ is closed in $\pi^{-1}(\beta)$. By \cite{rcs}, the specializations of any point, and specifically of $\sigma(\alpha) \in \sper A[t]$, form a chain ordered by inclusion, and so we must have 
\[  \sigma(\alpha) \subseteq \gamma \subseteq \sigma(\beta).  \]
Let $C_{\gamma \sigma(\alpha)}$ be the largest convex subring of $\rho(\sigma(\alpha))=\rho(\alpha)$ with a maximal ideal $m_{\gamma}$ such that 
\[  \rho_{\alpha}^{-1}(m_{\gamma})=\supp(\gamma)/\supp(\sigma(\alpha)). \]
We know $\supp(\gamma) \subseteq \supp(\sigma(\beta))$, and so the largest convex subring $C_{\sigma(\beta)\sigma(\alpha)}$ with a maximal ideal $m_{\sigma(\beta)}$ must contain the largest convex subring $C_{\gamma \sigma(\alpha)}$ with maximal ideal $m_\gamma$. Since convex subrings of real closed fields are valuation rings by (\cite{valuationsRCR}, 6) , It follows that $\gamma \in \overline{\sigma(\beta)}$, implying they are equal, and that specializations are unique in $\pi^{-1}(\beta)$. 
\par Now, we will show that if $\sigma:K \to \pi^{-1}(K)$ is a morphism of spectral spaces such that for all $\alpha,\beta \in K$ with $\alpha \subseteq \beta$,  $\sigma(\beta)$ is the only specialization of $\sigma(\alpha)$ in $\pi^{-1}(\beta)$, then $\sigma$ is a compatible section. For $\alpha$ and $\beta$ with $\beta \supseteq \alpha$, we have a commutative diagram
\begin{center}
\begin{tikzcd}
A[t]/\supp(\sigma(\alpha)) \ar[d,"s" swap] \ar[r,hook]  & C_{\sigma(\beta) \sigma(\alpha)} \ar[r,hook] &  \rho(\alpha) \\ A/\supp(\alpha) \ar[r,hook]  & C_{\beta \alpha}. \ar[ru,hook] & 
\end{tikzcd}
\end{center}
Since both $C_{\beta \alpha}$ and $C_{\sigma(\beta) \sigma(\alpha)}$ are convex subrings of $\rho(\alpha)$ with maximal ideals that have inverse image equal to 
\[  \supp(\sigma(\beta))/\supp(\sigma(\alpha)) \subset A[t]/\supp(\sigma(\alpha))  \]
under this diagram, 
\[  C_{\sigma(\beta) \sigma(\alpha)} \subseteq C_{\beta \alpha}. \]
Arguing towards a contradiction, suppose 
\[  C_{\sigma(\beta) \sigma(\alpha)} \neq C_{\beta \alpha}. \]
Let
\begin{center}
\begin{tikzcd}
A[t] \ar[r,two heads] & A[t]/\supp(\sigma(\alpha)) \ar[r,hook] & C_{\beta \alpha} \ar[r,two heads] & \rho_{\beta \alpha}
\end{tikzcd}
\end{center}
be a homomorphism mapping $t$ to an element of $C_{\beta \alpha} \setminus C_{\sigma(\beta) \sigma(\alpha)}$. Then
\begin{center}
\begin{tikzcd}
A[t] \ar[r,two heads] & A[t]/\supp(\sigma(\alpha)) \ar[r,hook] \ar[rd,hook] & C_{\beta \alpha} \ar[r,two heads] & \rho_{\beta \alpha} \\  & &  C_{\sigma(\beta)\sigma(\alpha)} \ar[r,two heads] \ar[u,hook] &  \rho_{\sigma(\beta)\sigma(\alpha)}
\end{tikzcd}
\end{center}
must yield a different specialization, contradicting our assumptions. Hence
\[  C_{\sigma(\beta) \sigma(\alpha)} = C_{\beta \alpha}, \]
and as a consequence,
\begin{center}
\begin{tikzcd}
& C_{\beta \alpha}=C_{\sigma(\beta) \sigma(\alpha)} \ar[d,hook] \ar[rd,"\pi_{\sigma(\beta) \sigma(\alpha})"] & \\ A[t]/\supp(\sigma(\alpha)) \ar[d] \ar[r,"\rho_{\sigma(\alpha)}"] \ar[ru,hook]  & \rho(\alpha)  & \rho_{\sigma(\beta) \sigma(\alpha)} \\  A[t]/\supp(\sigma(\beta)) \ar[r,"\rho_{\sigma(\beta)}"]  & \rho(\beta) \ar[ru,dotted,"\iota" swap] &
\end{tikzcd}
\end{center}
commutes for all $\alpha \in K$ with specialization $\beta \in K$. We have thus proved $\sigma$ is compatible, and we are done.
\end{proof}

This theorem is powerful, as it is half of the work necessary to establish $\mathscr{O}_{X}(K)$ is a ring. As an application, suppose $\sigma:K \to \pi^{-1}(K)$ is compatible as before, and further suppose 
\[ P(s):=\{\alpha \in K:\sigma(\alpha) \geq 0 \} \]
is constructible in $K$. Let $\beta$ be a specialization of $\alpha$. We then have 
\[  \iota \circ \rho_{\sigma(\beta)}(t) =\pi_{\sigma(\beta) \sigma(\alpha)} \circ \rho_{\sigma(\alpha)}(t), \]
and since $\iota$ and $\pi_{\sigma(\beta) \sigma(\alpha)}$ are both order preserving, 
\[  \sigma(\beta)=\rho_{\sigma(\beta)}(t) \geq 0.  \]
Hence $\beta \in P(s)$, and since $P(s)$ is closed under specialization and constructible, $P(s)$ is closed.
 
The other half of proving $\mathscr{O}_{X}(K)$ is a ring lies in establishing a clean way to work with morphisms with constructible image. As seen previously, working with constructible subsets is simplest with the machinery of model theory, so that is the machinery we shall use;

\begin{defn}
(\cite{rcs},8) Let $\pi:\sper A[t] \to \sper A$ be induced by the $A$-algebra structure homomorphism, and let $K \subseteq \sper A$ be (pro)constructible. A local section $\sigma:K \to \pi^{-1}(K)$ of $\pi$ is a \textit{constructible section} if $\sigma(K)$ is constructible in $\pi^{-1}(K)$
\end{defn}

\begin{prop}
(\cite{rcs},8) A local section $\sigma:K \to \pi^{-1}(K)$ of $\pi$ is constructible if and only if there is some formula $\Phi(T)$ in the language $\mathscr{L}_{\leq}(A)$ with variable $T$ such that 
\begin{itemize}
\item $\Phi(\sigma(\alpha))$, and
\item $ \exists x:\Phi(x) \, \wedge \, (\,\forall y \text{ such that } \Phi(y) \implies x=y ) $ 
\end{itemize}
holds for all $\alpha \in K$.
\end{prop}

\begin{proof}
First, suppose $\sigma:K \to \pi^{-1}(K)$ is a constructible section of $\pi$. Let \[ \phi(p_1(t),\hdots , p_r(t)) \]
be a sentence in $\mathscr{L}_{\leq}(A[t])$ defining $\sigma(K) \subseteq \pi^{-1}(K)$. Let $\Phi$ be the sentence
\[  \exists T:\phi(p_{1}(T),\hdots ,p_{r}(T)) \]
in $\mathscr{L}_{\leq}(A)$ with free variable $T$. We then have 
\[ \Phi(\sigma(\alpha))=\phi(p_{1}(\sigma(\alpha)),\hdots ,p_{r}(\sigma(\alpha)))=\phi(p_{1}(\rho_{\sigma(\alpha)}(t)),\hdots ,p_{r}(\rho_{\sigma(\alpha)}(t))), \]
which is a true sentence in $\rho(\alpha)=\rho(\sigma(\alpha))$. We also know $\rho_{\sigma(\alpha)}(t) \in \rho(\alpha)$ is the unique element of $\rho(\alpha)$ satisfying $\Phi$, since the image of $\alpha$ under $\sigma$ uniquely specifies the image of $t$ under $A[t] \to \rho(\alpha)$. 
\par Now suppose there is a sentence $\Phi(T)$ in the language $\mathscr{L}_{\leq}(A)$ with variable $T$ such that 
\begin{itemize}
\item $\Phi(\sigma(\alpha))$, and
\item $ \exists x:\Phi(x) \, \wedge \, (\,\forall y \text{ such that } \Phi(y) \implies x=y ) $ 
\end{itemize}
holds for all $\alpha \in K$. We hope to prove $\sigma(K)$ is constructible in $\pi^{-1}(K)$, so it suffices to find a sentence $\phi$ in $\mathscr{L}_{\leq}(A[t])$ which describes $\sigma(K)$. Define $\phi(t)$ in the language $\mathscr{L}_{\leq}(A[t])$ to be $\Phi(t)$. The two defining properties of $\Phi(T)$ are inherited by $\phi(t)$. The first property of $\phi(t)$ guarantees that for each $\alpha$, $\phi(\sigma(\alpha))$ holds, and second property assures us that $\sigma(\alpha)$ is the unique input for which $\phi(t)$ holds. Hence $\phi(t)$ describes $\sigma(K)$ and the claim is proven.
\end{proof}

We see the sentence 
\[ \Phi_{a}(T):\exists T \,\text{ such that }\, T-a=0 \]
in the language $\mathscr{L}_{\leq}(A)$ describes the image of the $a \in A$ section, so by Proposition 27 it is constructible. If $\sigma_1,\sigma_2$ are constructible sections with images defined by sentences $\Phi_1(T),\Phi_2(T)$, the sentence
\[ \exists X \, \exists Y:T=X+Y \wedge \Phi_1(X) \wedge \Phi_2(Y)  \]
defines their sum by Proposition 27, and a similar trick works for multiplication. We may conclude

\begin{cor}
(\cite{rcs},9) The set of constructible sections $\sigma:K \to \pi^{-1}(K)$ forms a unital subring of $\prod_{\alpha \in K}\rho(\alpha)$.
\end{cor}

Finally, we conclude with a proof that $\mathscr{O}_{X}(K)$ is the ring of constructible and compatible sections $\sigma \in \prod_{\alpha \in K}\rho(\alpha)$. Since $\mathscr{O}_{X}$ is already a collection of functions on $X$, we may conclude the constructible and compatible sections form a sheaf, as a treat.

\begin{thm}
(\cite{rcs},9-10) For each (pro)constructible subspace $K \subseteq X$, the image of the map $i:\mathscr{O}_{X}(K) \to \prod_{\alpha \in K}\rho(\alpha)$ is the ring of constructible and compatible sections.
\end{thm}
\begin{proof}
Let $\sigma \in \mathscr{O}_{X}(K)$, and suppose $\alpha,\beta \in K$ with $\beta \in \overline{\alpha}$. Since $\sigma:K \to \pi^{-1}(K)$ is continuous in the weak and constructible topologies, $\sigma(\beta)$ is a specialization of $\sigma(\alpha)$. I claim $\sigma(\beta)$ is the unique specialization of $\sigma(\alpha)$ in $\pi^{-1}(\beta)$. If $\gamma \in \pi^{-1}(\beta)$ is any specialization of $\sigma(\alpha)$, we know $\sigma(K)$ is closed, which implies $\gamma \in \sigma(K)$. Consequently,
\[ \gamma \in \sigma(K) \cap \pi^{-1}(\beta)=\{\sigma(\beta)\}, \]
and $\sigma(\beta)$ is the unique specialization of $\sigma(\alpha)$ in $\pi^{-1}(\beta)$. By Theorem 25, $\sigma$ is compatible, and $\sigma$ is constructible by definition, so we have proved $i(\mathscr{O}_{X}(K))$ is contained in the ring of constructible and compatible sections.
\par It is left to show that each constructible, compatible section is contained in the image of $i$. Let $\sigma \in \prod_{\alpha \in K}\rho(\alpha)$ be a constructible and compatible section, which we can certainly view as a map $\sigma:K \to \pi^{-1}(K)$ which sends $\alpha$ to $\pi_{\alpha}(\sigma) \in \rho(\alpha)$. As before, $\sigma(K)$ is constructible by definition. The morphism 
\[ \pi\biggr\lvert_{\sigma(K)}:\sigma(K) \to K \]
is bijective and continuous with respect to both the weak and constructible topologies, and with respect to the latter, $\pi\lvert_{\sigma(K)}$ is a homeomorphism. Since $\sigma:K \to \sigma(K)$ is the inverse of $\pi\lvert_{\sigma(K)}$, $\sigma:K \to \pi^{-1}(K)$ is continuous in the constructible topology. We still need to show that $\sigma(K)$ is closed, and that $\sigma:K \to \pi^{-1}(K)$ is continuous in the weak topology.
\par Let $C$ be a closed constructible subset of $\pi^{-1}(K)$. We know $\sigma^{-1}(C)$ is constructible, and it is left to show $\sigma^{-1}(C)$ is closed under specialization. Suppose $\alpha$ is in $\sigma^{-1}(C)$ with specialization $\beta$. Then $\sigma(\beta) \supseteq \sigma(\alpha)$, and since $C$ is closed, $\sigma(\beta) \in C$. Hence $\beta \in \sigma^{-1}(C)$, implying $\sigma^{-1}(C)$ is closed constructible and $\sigma:K \to \pi^{-1}(K)$ is continuous in the weak topology.
\par To show $\sigma(K)$ is closed, it suffices by constructibility to show $\sigma(K)$ is closed under specialization. Suppose $\alpha \in K$, and let $\gamma$ be a specialization of $\sigma(\alpha)$. Since $\{\sigma(\alpha)\}$ is closed in $\pi^{-1}(\alpha)$, if $\pi(\gamma)=\alpha$, then $\gamma=\sigma(\alpha)$. On the other hand, if $\pi(\gamma) \neq \alpha$, we still need to show $\gamma$ is in the image of $\sigma$. Since $\gamma \supseteq \sigma(\alpha)$, we know $\pi(\gamma) \supseteq \alpha$, i.e. $\pi(\gamma)$ is a specialization of $\alpha$. By Theorem 25, $\sigma(\pi(\gamma))$ is the unique specialization of $\sigma(\alpha)$ in $\pi^{-1}(\pi(\gamma))$, and so $\beta=\sigma(\pi(\beta))$ and we may conclude this proof. 
\end{proof}

From here onward, we will refer to $\mathscr{O}_{X}$ as the sheaf of constructible and compatible sections on $X=\sper A$, and Theorem 29 shows us that we have two useful ways in which to think of this sheaf.

\begin{defn}
Let $X=\sper A$. The \textit{real closure} of $A$ is the ring $\mathscr{O}_{X}(X)$. For a (pro)constructible subset $K \subseteq X$, the \textit{real closure of $A$ on $K$} is the ring $\mathscr{O}_{X}(K)$. A ring $A$ is \textit{real closed} if it is isomorphic to its real closure.
\end{defn}

\subsection{Real Closed Rings are (co)Complete}

We now turn towards the fundamental categorical properties of real closed rings, in an attempt to prove that the category of real closed rings with ring homomorphisms contains all limits and all colimits over small diagrams. 

\begin{defn}
(\cite{semiAlgFuncNreflectors}, Section 2, 35) Let $\mathcal{C},\mathcal{S}$ be two categories with $\mathcal{S}$ a full subcategory of $\mathcal{C}$. Let $\iota:\mathcal{S} \to \mathcal{C}$ be the inclusion functor. If $\iota$ has a left adjoint $r:\mathcal{C} \to \mathcal{S}$, i.e. for all $A \in \text{Ob}(\mathcal{C})$ and $B \in \text{Ob}(\mathcal{S})$
\[ \text{Hom}_{\mathcal{C}}(A,\iota(B))=\text{Hom}_{\mathcal{S}}(r(A),B),  \]
We call $r$ a \textit{reflector}, and $\mathcal{S}$ a \textit{reflective subcategory} of $\mathcal{C}$. 
\end{defn}

Being a reflective subcategory is a surprisingly strong condition; 

\begin{prop}
(\cite{semiAlgFuncNreflectors}, Section 2, 38-39) Let $r:\mathcal{C} \to \mathcal{S}$ be a reflector for the reflective subcategory $\mathcal{S}$. If $\mathcal{C}$ is cocomplete, then so is $\mathcal{S}$. 
\end{prop}
\begin{proof}
Left adjoints preserve colimits.
\end{proof}

We cannot, on the other hand, push limits from $\mathcal{C}$ down to limits of $\mathcal{S}$, except for in special cases.

\begin{prop}
(\cite{semiAlgFuncNreflectors},38-39) Let $r:\mathcal{C} \to \mathcal{S}$ is a reflector such that for each object $A$ in $\mathcal{C}$, each reflection $r_{A}:A \to r(A)$ is an epimorphism. If $\mathcal{C}$ is complete, then so is $\mathcal{S}$.  
\end{prop}
\begin{proof}
Let $\mathcal{F}:J \to \mathcal{C}$ be a diagram with small index category $J$. Since $\mathcal{C}$ is complete,
\[ \lim_{j \in J} \mathcal{F}(j)  \]
is in $\mathcal{C}$. Composing $r$ with $\mathcal{F}$ gives a diagram $r\mathcal{F}:J \to \mathcal{S}$ in $\mathcal{S}$, and $\iota \circ r\mathcal{F}:J \to \mathcal{C}$ also has a limit 
\[  \lim_{j \in J}r\mathcal{F}(j)   \]
in $\mathcal{C}$. We will show that this limit lies in $\mathcal{S}$.
\par Showing 
\[  \lim_{j \in J}r\mathcal{F}(j) \in \mathcal{S}  \]
is equivalent to showing that the map
\[ \hat{r}:\lim_{j \in J}r\mathcal{F}(j) \to r\left(\lim_{j \in J}r\mathcal{F}(j)\right)  \]
is an isomorphism. For ease, let \[ L:=\lim_{j \in J}r\mathcal{F}(j).   \]
For each $j \in J$, we have a diagram
\begin{center}
\begin{tikzcd}
 L \ar[d] \ar[rd,"\hat{r}"] &  \\ r(\mathcal{F}(j))  &   rL,
\end{tikzcd}
\end{center}
and since $r$ is an adjoint to the inclusion functor, we get a unique map
\begin{center}
\begin{tikzcd}
 L \ar[d] \ar[rd,"\hat{r}"] &  \\ r(\mathcal{F}(j))  &   rL \ar[l,dotted,"\exists !"]
\end{tikzcd}
\end{center}
making the triangle commute. If $i \to j$ is a map in $J$, the universal property of the limit $L$ guarantees that 
\begin{center}
\begin{tikzcd}
 L \ar[d] \ar[rd] &   \\  r(\mathcal{F}(i)) \ar[r] &  r(\mathcal{F}(j))  
\end{tikzcd}
\end{center}
commutes. We also have a diagram
\begin{center}
\begin{tikzcd}
 rL \ar[d] \ar[rd] &   \\  r(\mathcal{F}(i)) \ar[r] &  r(\mathcal{F}(j)),  
\end{tikzcd}
\end{center}
but since we only know $rL$ maps to $r(\mathcal{F}(i))$ for all $i \in I$, and not that it respects the maps in the diagram $r\mathcal{F}$, we need to demonstrate this fact. The diagram 
\begin{center}
\begin{tikzcd}
L \ar[r,"\hat{r}"] & rL \ar[d] \ar[rd] &   \\ & r(\mathcal{F}(i)) \ar[r] &  r(\mathcal{F}(j))  
\end{tikzcd}
\end{center}
commutes, and since $\hat{r}$ is an epimorphism, the original triangle 
\begin{center}
\begin{tikzcd}
 rL \ar[d] \ar[rd] &   \\  r(\mathcal{F}(i)) \ar[r] &  r(\mathcal{F}(j)) 
\end{tikzcd}
\end{center}
commutes. We have now shown $rL$ maps to the limit $L$ of the diagram $r\mathcal{F}$, and by the universal property of the limit, this map is an isomorphism and $L$ is in $\mathcal{S}$.
\end{proof}

We fix a notation for reflectors $r:\mathcal{C} \to \mathcal{S}$ as in Proposition 32, and generalize the idea below:

\begin{defn}
A reflector $r:\mathcal{C} \to \mathcal{S}$ is called an epireflector (respectively monoreflector) if $r_A:A \to rA$ is an epimorphism (respectively monomorphism) for all $A \in \text{Ob}(\mathcal{C})$.
\end{defn}

\begin{prop}
All monoreflectors are epireflectors.
\end{prop}
\begin{proof}
Let $r:\mathcal{C} \to \mathcal{S}$ be a monoreflector, and let $\varphi_1,\varphi_2:rC \to D$ be $\mathcal{C}$-morphisms such that $\varphi_1 \circ r_C=\varphi_2 \circ r_C$. Consider the diagram
\begin{center}
\begin{tikzcd}
 C \ar[r,"\varphi_{i} \circ r_C"] \ar[d,"r_C"] & D \ar[d,"r_D"] \\ rC \ar[r] & rD.   
\end{tikzcd}
\end{center}
The morphism $rC \to rD$ is $r_{D} \circ \varphi_{i}$ by definition, but equals $r(\varphi_{i}\circ r_{C})$ since $r$ is a reflector. Hence 
\[ r_{D} \circ \varphi_{1}=r_{D} \circ \varphi_{2}, \]
and since $r_D$ is a monomorphism, $\varphi_1=\varphi_2$.
\end{proof}

Let $\redPo$ be the category of reduced partially ordered rings with ring homomorphisms. Let $\textbf{WRR}$ be the full subcategory $\cRing$ whose objects are rings $A$ with the following property; 
\[ \forall a_1, \hdots ,a_n \in A: \;\left(\sum_{i=1}^{n}a_{i}^{2}=0\right) \implies (a_{i}^{2}=0 \; \forall i) .   \]
 We call these rings \textit{weakly real} (\cite{semiAlgFuncNreflectors},22-23). Using $\sper(-)$, we may construct a functor 
\[  \mathscr{F}:\textbf{WRR} \to \redPo  \]
as follows; given a weakly real ring $A$, the cone 
\[ P_{w}(A):=\sum A^{2} \]
guarantees that $\sper(A) \neq \emptyset$. So, the product of the canonical morphisms $\rho_\alpha$ gives a morphism
\[  A \to \prod_{\alpha \in \sper(A)}\rho(\alpha)  \]
of partially ordered rings. Moving forward, we use the notation $\mathscr{F}(A)=\prod A$ for 
\[ \prod_{\alpha \in \sper(A)}\rho(\alpha). \]
If $f:A \to B$ is a ring homomorphism, I claim we obtain a commutative diagram 
\begin{center}
\begin{tikzcd}
A \ar[r,"f"] \ar[d] & B \ar[d] \\ \prod A \ar[r,"\prod (f)"] & \prod B
\end{tikzcd}
\end{center}
in rings. 

We must first define the morphism $\mathscr{F}(f)=\prod(f)$. Let \[ \prod(f)((a_{\alpha})_{\alpha \in \sper A})(\beta)=\iota_{f^{-1}(\beta)}(a_{f^{-1}(\beta)}). \]
$\prod(f)$ is a well-defined map of sets because it is given in terms of set maps. If 
\[  (a_{\alpha})_{\alpha \in \sper A},(b_{\alpha})_{\alpha \in \sper A} \in \prod A,  \]
we have 
\begin{align*}
& \prod(f)((a_{\alpha}+b_{\alpha})_{\alpha \in \sper A})(\beta) \\=&\iota_{f^{-1}(\beta)}(a_{f^{-1}(\beta)}+b_{f^{-1}(\beta)}) \\=& \iota_{f^{-1}(\beta)}(a_{f^{-1}(\beta)})+\iota_{f^{-1}(\beta)}(b_{f^{-1}(\beta)}) \\=& \prod(f)(a_{\alpha})(\beta)+\prod(f)(b_{\alpha})(\beta).
\end{align*}
A similar computation for multiplication shows that $\prod(f)$ is a ring homomorphism.

Now, to show 
\begin{center}
\begin{tikzcd}
A \ar[r,"f"] \ar[d] & B \ar[d] \\ \prod A \ar[r,"\prod (f)"] & \prod B
\end{tikzcd}
\end{center}
commutes, fix $a \in A$. Checking commutativity is equivalent to checking that 
\[  \iota_{f^{-1}(\beta)}(a(f^{-1}(\beta)))=f(a)(\beta).  \]
Restricting our focus to the point $\beta$, the diagram
\begin{center}
\begin{tikzcd}
A \ar[d,"f" swap] \ar[r,"\rho_{f^{-1}(\beta)}"] & \rho(f^{-1}(\beta)) & \\ B \ar[r,"\rho_{\beta}" ] & \rho(\beta) & 
\end{tikzcd}
\end{center}
factors as 
\begin{center}
\begin{tikzcd}
A \ar[d,"f" swap] \ar[r] & A/\supp(f^{-1}(\beta)) \ar[d] \ar[r] & \rho(f^{-1}(\beta)) \\ B \ar[r] & B/\supp(\beta) \ar[r] & \rho(\beta).
\end{tikzcd}
\end{center}
Both factor rings are integral domains, so we may factor our diagram through their ordered fields of fractions;
\begin{center}
\begin{tikzcd}
A/\supp(f^{-1}(\beta)) \ar[d] \ar[r] & \kappa(f^{-1}(\beta)) \ar[d] \ar[r] & \rho(f^{-1}(\beta)) \\ B/\supp(\beta) \ar[r] & \kappa(\beta) \ar[r] & \rho(\beta).
\end{tikzcd}
\end{center}
The universal property of real closed fields gives a unique map $\rho(f^{-1}(\beta)) \to \rho(\beta)$ such that 
\begin{center}
\begin{tikzcd}
A/\supp(f^{-1}(\beta)) \ar[d] \ar[r] & \kappa(f^{-1}(\beta)) \ar[d] \ar[r] & \rho(f^{-1}(\beta)) \ar[d,dotted] \\ B/\supp(\beta) \ar[r] & \kappa(\beta) \ar[r] & \rho(\beta)
\end{tikzcd}
\end{center}
commutes, and so $\iota_{f^{-1}(\beta)}$ must be that unique map!

Subrings of $\prod A$, the total ring of functions on $A$, play a crucial role in establishing the real closure operation as a reflector. We note that $\prod(-)$ is not a reflector, as it is not generally idempotent (\cite{semiAlgFuncNreflectors}, 55-56).

Let $\R_{alg}$ be the field of real algebraic numbers, and let $\rho$ be any real closed field. Since $\Z$ is initial in the category of commutatve rings, there is a unique morphism $\Z \to \rho$. Since $\rho$ is real closed, we have a factorization 
\begin{center}
\begin{tikzcd}
\Z \ar[r] \ar[rd] & \rho \\ & \Gamma(\sper(\Z)) .\ar[u,"\exists !" dotted, swap]
\end{tikzcd}
\end{center}
Now, $\sper(\F_p)=\emptyset$ for characteristic reasons, so 
\[ \sper(\Z)=\bigcup_{p \text{ prime}}\emptyset \cup\sper(\Q)=\sper(\Q).  \]
Since an ordered field structure on $\Q$ must come from an ordering of $\R$, the real closure of $\Q$ is $\R_{alg}$. So every real closed field $\rho$ is a $\R_{alg}$ algebra. We fittingly denote $\R_{alg}=\R_{0}$ moving forward.

Let $\sigma(n)$ denote the ring of semialgebraic functions in $n$ variables. Given a semialgebraic function $\varphi:\R_{0}^{n} \to \R_0$, we may think of $\varphi$ as an $n$-ary operation on $\R_{0}$. Furthermore, if $g_1,\hdots,g_n:\R_{0}^{k} \to \R_{0}$ are semialgebraic, the operation
\[ \omega_{\varphi}:\sigma(k)^{n} \to \sigma(k) \]
which maps $(g_1,\hdots,g_n) \mapsto \varphi(g_1,\hdots,g_n)$
is an $n$-ary operation on $\sigma(k)$ for every natural $k$. We also know a set map $\varphi:\R_{0}^{n} \to \R_{0}$ is semialgebraic if and only if $\varphi(g_1,\hdots,g_n)$ is semialgebraic for all $n$-tuples of semialgebraic functions $(g_1,\hdots,g_n)$.  Hence semialgebraic functions correspond bijectively to $n$-ary operations on semialgebraic functions.

Given a semialgebraic function $\varphi:\R_{0}^{n}\to \R_{0}$ and any total ring of functions $\prod A$, we obtain an $n$-ary operation on $\prod A$ as follows; since 
\[ \prod A=\prod_{\alpha \in \sper(A)}\rho(\alpha), \]
and $\rho(\alpha)$ is a $\R_0$-algebra, we may form the tensor product
\[ \varphi \otimes_{\R_0} 1_{A}:\R_{0}^{n} \otimes \prod A \to \R_{0} \otimes \prod A. \]
The multiplication operation in $\prod A$ allows us to write the map as 
\[  \varphi_{A}:\left(\prod A\right)^{n}   \to \prod A, \]
so $\varphi$ also gives an $n$-ary operation on total rings of functions. 

We have a notion of '$n$-ary operations on $A$' (\cite{semiAlgFuncNreflectors},65) which does not align with the usual use of those words. Let 
\[  \Delta_{A}:A \to \prod A   \]
be the canonical map from $A$ to its total ring of functions. Let $\omega$ be an $n$-ary operation on $\prod A$. The composition 
\begin{center}
\begin{tikzcd}
 A^{n}  \ar[r,"\Delta_{A}^{\times n}"] & \left(\prod_{\alpha \in \sper A}\rho(\alpha)\right)^{n} \ar[r,"\omega"] & \prod_{\alpha \in \sper A}\rho(\alpha)
\end{tikzcd}
\end{center}
is a map of $A$-algebras such that 
\[  \omega(\Delta_{A}^{\times n}(a_1,\hdots,a_n))(\beta)=\omega(a_1(\beta),\hdots,a_n(\beta)). \]

\begin{defn}
Let $A$ be a weakly real ring. An element $f \in \prod A$ is a \textit{semialgebraic function on $A$} if there exists a natural $n$, $\{ a_1,a_2,\hdots ,a_n \} \subseteq A$ and an $n$-ary operation $\omega$ such that 
\[  f(\alpha)=\omega(\Delta_{A}^{\times n}(a_1,\hdots,a_n))(\alpha)  \]
for all $\alpha \in \sper A$.
\end{defn}

We denote the semialgebraic functions on A by $\semf(A)$. $n$-ary operations are well-behaved under ring homomorphisms; 

\begin{prop}
(\cite{semiAlgFuncNreflectors},66) If $f:A \to B$ is a ring homomorphism of weakly real rings, and $\omega$ is an $n$-ary operation, then
\begin{center}
\begin{tikzcd}
A^n \ar[r,"f^n"] \ar[d,"\omega" swap] & B^n \ar[d,"\omega"] \\ \prod A \ar[r,"\prod(f)"] & \prod B
\end{tikzcd}
\end{center}
commutes in \underline{\textbf{Sets}}.
\end{prop}
\begin{proof}
Indeed, for each $\beta \in \sper B$ and $a_1,\hdots,a_n \in A$ we have
\begin{align*}
&\prod(f)(\omega(a_1,\hdots,a_n))(\beta) \\=&\iota_{f^{-1}(\beta)}(\omega(a_1,\hdots,a_n)(f^{-1}(\beta))) \\=& \iota_{f^{-1}(\beta)}(\omega_{f^{-1}(\beta)}(a_1(f^{-1}(\beta)),\hdots,a_n(f^{-1}(\beta))) \\=& 
\omega_{f^{-1}(\beta)}(\iota_{f^{-1}(\beta)}(a_1(f^{-1}(\beta))),\hdots,\iota_{f^{-1}(\beta)}(a_n(f^{-1}(\beta)))) \\=& \omega_{\beta}(f(a_1)(\beta),\hdots,f(a_n)(\beta))  \\=&\omega(f(a_1),\hdots,f(a_n))(\beta).
\end{align*}
\end{proof}

In the above proof, we may note that the semialgebraic functions $\omega(a_1,\hdots,a_n)$ are mapped to semialgebraic functions $\omega(f(a_1),\hdots,f(a_n))$ by $\prod(f)$, for any ring homomorphism $f:A \to B$. Hence $f$ gives a map $\semf(f):\semf(A) \to \semf(B)$.

Semialgebraic functions are also preserved under $n$-ary operations;
\begin{prop}
(\cite{semiAlgFuncNreflectors},67) Suppose $A$ is a weakly real ring, $\omega:\left(\prod A \right)^{n} \to \prod A$ is an $n$-ary operation, and $\omega_1,\hdots,\omega_i,\hdots,\omega_n$ are semialgebraic functions on $A$ parameterized by $(a_{i,j})_{1 \leq i \leq n ,\, 1 \leq j \leq r_j}$. Then \[ \omega(\omega_1(a_{1,1},\hdots,a_{1,n_1}),\hdots,\omega_n(a_{n,1},\hdots,a_{n,r_n})) \]
is semialgebraic.
\end{prop}
\begin{proof}
Let $\Phi(X_1,\hdots,X_n,Y)$ be a first-order formula describing the graph of $\omega$, and let $\Phi_{i}(X_{i,1},\hdots,X_{i,n_i},X_i)$ be a first order sentence describing the graph of $\omega_i$. Let $\Psi(X_{1,1},\hdots,X_{n,r_n},Y)$ be the sentence
\[  \exists X_{1,1},\hdots,X_{n,r_n}:\left(\bigwedge_{i=1}^{n} \Phi_{i}(X_{i,1},\hdots,X_{i,n_i},X_i)\right) \wedge \Phi(X_1,\hdots,X_n,Y). \]
The sentence $\Psi$ is a first order sentence which also describes the graph of a semialgebraic function; in fact, it describes
\[ \omega(\omega_1(a_{1,1},\hdots,a_{1,n_1}),\hdots,\omega_n(a_{n,1},\hdots,a_{n,r_n})), \]
concluding this proof
\end{proof}

In fact, the $\text{semf}(\cdot)$ operation is idempotent, i.e. the natural map $\semf(A) \to \semf(\semf(A))$ is an isomorphism (\cite{semiAlgFuncNreflectors},75). More strongly, it can be shown that $\semf:\redPo \to \textbf{\underline{semf}}$ is a monoreflector. The fact that $A \to \semf(A)$ is a monomorphism follows from the factorization
\begin{center}
\begin{tikzcd}
A \ar[r,hook] \ar[rd] & \prod A \\ & \semf(A) \ar[u,hook],
\end{tikzcd}
\end{center}
where the top morphism is the canonical map from $A$ to its total ring of functions $\prod A$. To show $\semf(\cdot)$ is a left adjoint of the forgetful functor, by (\cite{semiAlgFuncNreflectors},71-72), if $A$ is a reduced weakly real ring, the map
\[ A \to \semf(A) \]
is an epimorphism in $\redPo$.

\begin{thm}
(\cite{semiAlgFuncNreflectors},89) Semialgebraic functions are the smallest monoreflective subcategory of $\redPo$.    
\end{thm}
\begin{proof}
To prove this theorem, we will first prove a key lemma;
\begin{lemma}
If $r:\mathscr{C} \to \mathscr{D}$ is a fixed monoreflector and $E$ is an epicomplete object, then 
\[ r_{E}:E \to rE \]
is an isomorphism.
\end{lemma}
\begin{proof}
Suppose $r_{E}:E \to rE$ is an epimorphism and a monomorphism, with $E$ an object of $\mathscr{D}$. Since $r$ is a left adjoint to the forgetful functor, we have a morphism
\[   \eta_{E}:rE \to E \]
corresponding to the identity morphism $E \to E$ via adjunction. Consider the sequence of morphisms
\begin{center}
\begin{tikzcd}
E \ar[r,"r_E"] & rE \ar[r,"\eta_E"] & E.
\end{tikzcd}
\end{center}
We note that $r_E$ is induced by the identity, and since $r_E$ is an epimorphism, any other morphism $rE \to E$ must be induced by the identity. Hence 
\[  \eta_{E} \circ r_{E}=\text{id}_{E}.  \]
Now for $rE$, we have a corresponding diagram
\begin{center}
\begin{tikzcd}
rE \ar[r,"\eta_E"] & E \ar[r,"r_E"] & rE.
\end{tikzcd}
\end{center}
The correspondence of morphisms under adjunction is bijective, and since $r_E$ corresponds to \emph{both} $\text{id}_{rE}$ and $r_E \circ \eta_E$,
\[ \text{id}_{rE}=r_E \circ \eta_E  \]
and $r_E$ is an isomorphism.
\end{proof}
Hence every monoreflective subcategory of $\redPo$ must contain the epicomplete objects in $\redPo$. By (\cite{semiAlgFuncNreflectors},76), Corollary 7.20, $\textbf{\underline{semf}}$ equals the epicomplete objects in $\redPo$, proving the theorem we set out to show.
\end{proof}

Semialgebraic functions form a full subcategory of $\redPo$, and if we restrict further to continuous semialgebraic functions, we get another full subcategory of $\redPo$. We can show that both of these subcategories are reflective subcategories of $\redPo$. The easiest way to give the reflector is by first introducing the definition of \textit{H-closed} reflectors, which allows us define reflectors solely by specifying what they do on the reduced partially ordered rings $\Z[x_1,\hdots,x_n]$. We will conclude by showing that $\Z[x_1,\hdots,x_n]$ maps via epimorphism to its real closure, meaning the real closure is an H-closed monoreflector of $\redPo$.

\begin{defn}
A full isomorphism closed subcategory $\mathscr{D}$ of a concrete category $ \mathscr{C}$ is \textit{H-closed} if for all $X \in \text{ob}(\mathscr{D})$ and for all $Y \in \text{ob}(\mathscr{C})$ with a surjection $X \twoheadrightarrow Y$, $Y \in \text{ob}(\mathscr{D})$.
\end{defn}

Some examples and non-examples of H-closed reflectors are given on pages 119-123 of \cite{semiAlgFuncNreflectors}, though we focus on the real closure reflector moving forward. Let
\[ \rho:\redPo \to \rcr  \]
be the functor between full subcategories of \cRing which takes a reduced partially ordered ring (with non-empty real spectrum) $A$ to the ring of continuous semialgebraic functions $\rho A$ on $A$. To verify this is a functor, we need an assignment of morphisms which respects composition and identity. If $f \in \rho A$ and $\varphi:A \to B$ is a ring homomorphism, let 
\[ f=\omega_f(a_1,\hdots,a_n), \] 
where $\omega_f$ is an $n$-ary operation and $a_1,\hdots,a_n \in A$. By Proposition 34,
\[ \varphi(f)=\omega_{f}(\varphi(a_1),\hdots,\varphi(a_n)),  \]
and since $\omega_f$ is an $n$-ary operation corresponding to a continuous semialgebraic function, the function $\varphi(f)$ is continuous and semialgebraic. Since the total ring of functions is a functor whose source category is any subcategory of weakly real rings, it preserves composition and identity, and consequently so does $\rho$. 

To show $\rho$ is a reflector is equivalent to showing every map $\varphi$ from a ring $A$ to a real closed ring $B$ factors uniquely through the real closure of $A$. The proof is to simply show $\rho$ is a left adjoint to the forgetful functor via commutativity of the diagram
\begin{center}
\begin{tikzcd}
A \ar[r,"\varphi"] \ar[d,"\rho_{A}" swap] & B \ar[d,"\cong"] \\ \rho A \ar[r,"\rho(\varphi)"] & \rho B    
\end{tikzcd}
\end{center}
so we omit the details. Since we have a factorization 
\[  A \to \rho A \to \semf(A)  \]
and $\underline{\textbf{semf}}$ is monoreflective, $A \to \rho A$ is a monomorphism and $\rcr$ is a monoreflective subcategory of $\redPo$. 

In fact, we can $\rho$ is an H-closed reflector of $\redPo$. By (\cite{semiAlgFuncNreflectors},114-115), Theorem 10.6, H-closed reflectors of $\redPo$ are uniquely determined by how they reflect the partially ordered reduced rings 
\[ \{ \Z[x_1,\hdots,x_n] \}_{n \in \N},   \]
with each ordered by its 'sums of squares' positive cone.
\begin{prop}
Let $\mathscr{F}(n):=\Z[x_1,\hdots,x_n]$, together with the partial order given by sums of squares. Then the map 
\[\rho_{\mathscr{F}(n)}:\mathscr{F}(n) \to \rho\mathscr{F}(n)\]
is an epimorphism. 
\end{prop}
\begin{proof}
The central observation in this proof is that $\mathscr{F}(n)$ is \textit{already} a collection of continuous $n$-ary operations on any real closed field! Let $C \in \text{ob}(\redPo)$, and let 
\[ \varphi_1,\varphi_2:\rho\mathscr{F}(n) \to C \] 
be two partially ordered ring homomorphisms such that 
\[   \varphi_1 \circ \rho_{\mathscr{F}(n)}=\varphi_2 \circ \rho_{\mathscr{F}(n)}. \]
In particular, the morphisms $\varphi_1,\varphi_2$ map the variables $x_i$ to the same elements of $C$. Let $\omega(a_1,\hdots,a_m)$ be an arbitrary element of $\rho\mathscr{F}(n)$ with each $a_j(x_1,\hdots,x_n) \in \mathscr{F}(n)$. We then have 
\begin{align*}
&(\varphi_1-\varphi_2)\omega(a_1,\hdots,a_m) \\=&\varphi_1\omega(a_1,\hdots,a_m)-\varphi_2\omega(a_1,\hdots,a_m) \\=&\omega(\varphi_1(a_1(x_1,\hdots,x_n)),\hdots,\varphi_1(a_m(x_1,\hdots,x_n)))-\\& \omega(\varphi_2(a_1(x_1,\hdots,x_n)),\hdots,\varphi_2(a_m(x_1,\hdots,x_n))) \\=& \omega(a_1(\varphi_1(x_1),\hdots,\varphi_1(x_n)), \hdots,a_m(\varphi_1(x_1),\hdots,\varphi_1(x_n))) - \\& \omega(a_1(\varphi_2(x_1),\hdots,\varphi_2(x_n)), \hdots,a_m(\varphi_2(x_1),\hdots,\varphi_2(x_n))) \\=&0,
\end{align*}
proving that $\varphi_1=\varphi_2$ as $\mathscr{F}(n)$-module homomorphisms, and since they are ring maps, we've proved that $\rho_{\mathscr{F}(n)}$ is an epimorphism.
\end{proof} 

\begin{prop}
If $R$ is a real closed ring and $\varphi:R \to S$ is a surjective ring homomorphism, then $S$ is real closed.
\end{prop}
\begin{proof}
Since $\rho:\redPo \to \rcr$ is a reflector, we have a commutative diagram
\begin{center}
\begin{tikzcd}
0 \ar[r] & \ker(\varphi) \ar[r,hook] \ar[d] & R \ar[r,"\varphi" two heads] \ar[d,"\rho_{R}"] & S \ar[r] \ar[d,"\rho_{S}"] & 0 \\ 
0 \ar[r] & K_{\rho} \ar[r,hook]  & \rho R \ar[r,"\rho(\varphi)"] & \rho S \ar[r] & 0
\end{tikzcd}
\end{center}
in the category of $R$-modules, where $K_{\rho}$ is the kernel of $\rho(\varphi)$. We know $\rho_{R}$ is an isomorphism, so we may view $\ker(\varphi)$ as an ideal of $\rho R$. If we can show 
\[ \rho_{R}\biggr\lvert_{\ker(\varphi)}:\ker(\varphi) \to K_{\rho}  \]
is an isomorphism of $R$-modules, the Snake Lemma will show $\rho_S$ is an isomorphism, and the proof will be complete. 
\par We know $\rho_{R}|_{\ker(\varphi)}$ is an injective homomorphism because it is a composition of injective morphisms. It is left to show surjectivity of $\rho_{R}|_{\ker(\varphi)}$. Let $\omega(r_1,\hdots,r_k)$ be an element of $K_\rho$. We may consider $\omega(r_1,\hdots,r_k)$ as an element $r \in R$. We then have 
\begin{align*}
&\varphi(r) \\ =&\varphi(\rho_{R}^{-1}(\omega(r_1,\hdots,r_k))) \\ =& \varphi(\omega(\rho_{R}^{-1}(r_1),\hdots,\rho_{R}^{-1}(r_k))) \\ =&\omega(\varphi \circ \rho_{R}^{-1}(r_1),\hdots,\varphi \circ \rho_{R}^{-1}(r_k))   
\end{align*}
is in $\rho S$ since $\varphi \circ \rho_{R}^{-1}(r_i)$ is in $S$, so by commutativity, we have 
\[  \omega(\varphi \circ \rho_{R}^{-1}(a_1),\hdots,\varphi \circ \rho_{R}^{-1}(a_k))=0.  \]
Hence $r \in \ker(\varphi)$, implying $\rho_{R}|_{\ker(\varphi)}$ is an isomorphism of $R$-modules, and the Snake Lemma says $\rho_S$ is an $R$-module isomorphism as well.
\end{proof}

We have now shown $\rho:\redPo \to \rcr$ is an H-closed monoreflector of $\redPo$, so Propositions 31, 32, and 34 together give that $\rcr$ is complete and co-complete.


\section{Properties of Real Closed Rings}

We abbreviate $\mathscr{L}_{\leq}(A)$ as $\lorf A$. Since real closed rings are sections of the structure sheaf on $\sper A$, the following algebraic proofs use the geometry of $\sper A$ and the theory built up in the first section. 

\begin{prop}
(\cite{rcs}, 12) Let $K$ be a (pro)constructible subspace of $X=\sper(A)$. If $\sigma \in \ox(K)$ and $\sigma(\alpha) \neq 0$ in $\rho(\alpha)$ for all $\alpha \in K$, then $\sigma \in \ox(K)^{*}$.
\end{prop}
\begin{proof}
Define $\tau \in \prod_{\alpha \in K}\rho(\alpha)$ such that 
\[  t(\alpha):=\frac{1}{\sigma(\alpha)}. \]
Since $\sigma:K \to \pi^{-1}(K)$ is compatible and inverses are preserved under ring homomorphisms, $\tau:K \to \pi^{-1}(K)$ is also compatible. It is left to show $\tau$ is constructible.
\par Let $\Phi(T)$ be a sentence in $\lorf A$ such that $\Phi(\sigma(\alpha))$ and
\[ \exists x:\Phi(x) \wedge \, \forall \; y (\Phi(y) \implies y=x) \]
holds for all $\alpha \in K$. Let $\Psi(T)$ be the sentence 
\[ \exists T:1=Tx \wedge \Phi(x). \]
We certainly have $\Psi(\tau(\alpha))$ by definition of $\tau$. Additionally, in each $\rho(\alpha)$, there is an element $\tau(\alpha)$ such that $\Psi(\tau(\alpha))$. If there is an element $\omega \in \rho(\alpha)$ such that $\Psi(\omega)$, since $\sigma(\alpha)$ is the unique element satisfying $\Phi(T)$ in $\rho(\alpha)$,
\[ \omega\sigma(\alpha)=1, \]
which implies 
\[  \omega= \frac{1}{\sigma(\alpha)}=\tau(\alpha). \]
By Proposition 27, $\tau$ is constructible, and so $\tau$ is a constructible and compatible inverse to $\sigma$ in $\ox(K)$.
\end{proof}

\begin{prop}
(\cite{rcs}, 12) Let $\sigma \in \ox(K)$. If $\sigma(\alpha) \geq 0$ for all $\alpha \in K$, there is a section $\tau \in \ox(K)$ such that 
\begin{center}
\begin{itemize}
    \item $\tau(\alpha) \geq 0$ for all $\alpha \in K$, and
    \item $\sigma=\tau^{2}$.
\end{itemize}
\end{center}
\end{prop}
\begin{proof}
For all $\alpha \in K$, define
\[  \tau(\alpha):=+\sqrt{\sigma(\alpha)}. \]
Since the polynomial $x^{2}-\sigma(\alpha)$ has exactly two solutions in $\rho(\alpha)$, one of which is positive, $+\sqrt{\sigma(\alpha)}$ is the unique positive root of $\sigma(\alpha)$. So $\tau$ is well-defined, and a consequence of being defined by $\sigma$ is that $\tau$ is compatible. It is left to show $\tau$ is a constructible section.
\par Let Let $\Phi(T)$ be a sentence in $\lorf A$ such that $\Phi(\sigma(\alpha))$ and
\[ \exists x:\Phi(x) \wedge \, \forall \; y (\Phi(y) \implies y=x) \]
holds for all $\alpha \in K$. Let $\Psi$ be the sentence
\[  \exists T:T \geq 0 \wedge \Phi(T^{2}). \]
We then know $\Psi(\tau(\alpha))$ for each $\alpha$, and $\tau(\alpha)$ is unique because $\sigma(\alpha)$ has exactly one positive root. We have now proved the claim.
\end{proof}

\par Unlike the category of schemes, we can extend all the sections in our structure sheaf to the whole space:

\begin{thm}
(\cite{rcs}, 14) Let $L \subseteq K$ be a constructible subset of a (pro)constructible subspace $K \subseteq \sper A$. Let $\sigma \in \prod_{\alpha \in K}\rho(\alpha)$ be a compatible section such that 
\[  \sigma\biggr\lvert_{L} \in \ox(L), \]
and $\sigma(\alpha)=0$ for all $\alpha \in K \setminus L$. Then $\sigma \in \ox(K)$.
\end{thm}
\begin{proof}
By hypothesis, $\sigma \in \prod_{\alpha \in K}\rho(\alpha)$ is a compatible section, so if we can show $\sigma$ is a constructible section, we'll have our proof. Let $\Phi_L \in \lorf A$ be a sentence defining $L \subseteq K$. We know $\sigma \lvert_{L}$ is constructible by hypothesis, so by Proposition 27, there is a sentence $\Phi_{\sigma}(T)$ in $\lorf A$ with free variable $T$ such that $\sigma(\alpha)$ is the unique element of $\rho(\alpha)$ satisfying $\Phi_{\sigma}(T)$ for every $\alpha \in L$. The formula 
\[ (\neg \Phi_L \wedge T=0) \bigvee (\Phi_L \wedge \Phi_{\sigma}(T))  \]
defines $\sigma$, so $\sigma \in \ox(K)$.
\end{proof}

\section{Real Closed Spaces I}

\subsection{Objects}

\begin{defn}
(\cite{rcs}, 38-39) An \textit{affine real closed space} $X$ is a topological space isomorphic to $\sper A$ as a spectral space, together with the sheaf of real closed rings $\ox$. 
\end{defn}

\begin{prop}
Affine real closed spaces are locally ringed spaces.
\end{prop}
\begin{proof}
Let $X=\sper A$ be an affine real closed space. It is sufficient to show that 
\[ \ox\biggr\lvert_{\alpha}:=\colim_{V open, \alpha \in V} \ox(V) \]
is a local ring. Consider the diagram
\begin{center}
\begin{tikzcd}
 & \ox\biggr\lvert_{\alpha} & \\ \ox(X) \ar[ru] \ar[rr,"\psi"] \ar[rd,"ev_{\alpha}" swap] &   & \ox(X)_{\supp(\alpha)} \\ & \rho(\alpha). &
\end{tikzcd}
\end{center}
For all open $U$ containing $\alpha$, we get a map $ \ox(U) \to \ox(X)_{\supp(\alpha)}$ where we first extend the fucntions of $\ox(U)$ by zero to $\ox(X)$ to obtain a ring map $\ox(U) \to \ox(X) \to \ox(X)_{\supp(\alpha)}$. By the universal property of the colimit, we get map 
\begin{center}
\begin{tikzcd}
 & \ox\biggr\lvert_{\alpha} \ar[rd,dotted,bend left =20] & \\ \ox(X) \ar[ru] \ar[rr,"\psi"] \ar[rd,"ev_{\alpha}" swap] &   & \ox(X)_{\supp(\alpha)} \\ & \rho(\alpha), &
\end{tikzcd}
\end{center}
and applying that same property to $\text{ev}_{\alpha}$ gives a map \[ \ox\biggr\lvert_{\alpha} \to \rho(\alpha).  \]
The resulting triangles commute. 
Now let $f \in \ox(X)$ such that $f \notin \supp(\alpha)$. Since $ev_{\alpha}(f) \neq 0$ in $\rho(\alpha)$, and $ev_{\alpha}$ factors through $\ox\biggr\lvert_{\alpha}$, the image of $f$ in $\ox\biggr\lvert_{\alpha}$ is also nonzero. Since the ring map $\ox(X) \to \ox\biggr\lvert_{\alpha}$ sends ring elements in $\ox(X) \setminus \supp(\alpha)$ to nonzero elements of $\ox\biggr\lvert_{\alpha}$. The nonzero elements of $\ox\biggr\lvert_{\alpha}$ extend to nonzero sections $\eta$ on an open $V$ containing $\alpha$, and $\eta(\alpha) \neq 0$. Hence $\eta\biggr\lvert_{\alpha}$ is invertible, and the universal property of localization gives a commutative diagram
\begin{center}
\begin{tikzcd}
 & \ox\biggr\lvert_{\alpha} \ar[rd,dotted,bend left =15] & \\ \ox(X) \ar[ru] \ar[rr,"\psi"] \ar[rd,"ev_{\alpha}" swap] &   & \ox(X)_{\supp(\alpha)} \ar[lu,dotted,bend left=15] \\ & \rho(\alpha). &
\end{tikzcd}
\end{center}
Since the upper triangle commutes in both direction, it follows that both ring maps indicated by dotted arrows are isomorphisms inverse to each other.
\end{proof}

Let $f:\sper A \to \sper B$ be a morphism of spectral spaces with $\beta \in \sper B$ and $\alpha \in f^{-1}(\beta)$. Using the corresponding map of rings, we acquire a diagram
\begin{center}
\begin{tikzcd}
 B \ar[r] \ar[d] & A \ar[d]  \\ B_{\supp(\beta)}  &  A_{\supp(\alpha)}. 
\end{tikzcd}
\end{center}
Since $B \to A \to A_{\supp(\alpha)}$ takes the multiplicative system $B \setminus \supp(\beta)$ to \[ A \setminus f^{\#}(\supp(\beta)) = A \setminus \supp(\alpha),\]
the universal property of localization gives a ring morphism 
\begin{center}
\begin{tikzcd}
 B \ar[r] \ar[d] & A \ar[d]  \\ B_{\supp(\beta)} \ar[r,dotted] &  A_{\supp(\alpha)}. 
\end{tikzcd}
\end{center}
We have thus shown

\begin{lemma}
Morphisms of affine real closed spaces are morphisms of locally ringed spaces.
\end{lemma}

\begin{defn}
(\cite{rcs}, 39) A \textit{real closed space} $(X,\ox)$ is a locally ringed space together with an open cover $\{U_i\}_{i \in I}$ such that for each $i$, $(U_i,\mathscr{O}_{U_i})$ is an affine real closed space.
\end{defn}

Morphisms of real closed spaces are, of course, morphisms of locally ringed spaces, so we do not need to define real closed space morphisms separately. We suppress the notation $(X,\ox)$ moving forward, and will simply refer to the real closed space $X$.

\subsection{Subspaces}

Given a real closed space $X$, a topological subspace $S$ of $X$ which is subspace in the weak topology on $X$ can be given a structure sheaf $\mathscr{O}_{S}$ so that $(S,\mathscr{O}_{S})$ is a locally ringed space. Depending on how the topological subspace $S$ sits in the topology of $X$, we have different ways of speaking about subspaces, so allow me to elucidate this point below.

\begin{defn}
Let $X$ be a real closed space, and let $U$ be a topological subspace of $X$ which is open in $X$. Then the space $U$ together with the structure sheaf $\mathscr{O}_{U}:=\ox|_{U}$ is an \textit{open subspace} of $X$.
\end{defn}

\begin{ex}
Say $X=\sper \R[x,y]$, and in the language $\lorf \R[x,y]$, we consider the sentence $\Phi_{U}$
\[  \exists X,Y \text{ s.t. } \neg (X=0) \bigvee \neg (Y=0), \]
and let $U$ be the subspace of $X$ define by $\Phi_{U}$. If $\beta \in U$ and $\alpha$ is a generization of $\beta$, we have a  diagram
\begin{center}
\begin{tikzcd}
\R[x,y]/\supp(\alpha) \ar[r] \ar[d] & \rho(\alpha)  & \\ \R[x,y]/\supp(\beta) \ar[r,"\rho_{\beta}" ] & \rho(\beta). & 
\end{tikzcd}
\end{center}
Since $\rho(\beta) \models \Phi_{U}$, we know one at least one of $\bar{x}$ and $\bar{y}$ is not in $\ker(\rho_{\beta})$, and Without loss of generality, we assume $\bar{x} \notin \ker(\rho_{\beta})$. Since $\R[x,y]/\supp(\alpha) \to \R[x,y]/\supp(\beta)$ is surjective, $\bar{x}$ has a nonzero element $x+\supp(\alpha)$ in its preimage. Consequently, the composite map
\[ \R[x,y] \twoheadrightarrow \R[x,y]/\supp(\alpha) \hookrightarrow \R(\alpha) \]
does not have $x$ in its kernel, proving $\rho_{\alpha} \models \Phi_{U}$ Hence $U$ is closed under generization, proving $U$ is open. It is left to describe the relationship between the sheaves $\ox$ and $\mathscr{O}_{U}$.

We know the inclusion map $i:U \to X$ is continuous map in the weak and constructible topologies. On the level of stalks, we have a ring map
\[ \iota:\ox|_{i(\alpha)} \to \mathscr{O}_{U}|_{\alpha}, \]
which is the identity map from the definition of $\mathscr{O}_{U}$. Hence $i$ is a morphism of real closed spaces.

If $T$ is a real closed space with morphisms $\zeta_j:T \to U$ such that $i \circ \zeta_1 =i \circ \zeta_2$, we first observe that for each $\alpha \in T$, \[\zeta_1(\alpha)=\zeta_2(\alpha) \]
since $i$ is injective as a map of sets. For our structure sheaves, we have the commutative diagram
\begin{center}
\begin{tikzcd}
\ox\biggr\lvert_{i \circ \zeta_{j}(\alpha)} \ar[r,bend left=20,"\zeta_1^{\#}"] \ar[r,bend right=20,swap,"\zeta_2^{\#}"] & \mathscr{O}_{U}\biggr\lvert_{\zeta_{j}(\alpha)} \ar[r,"\iota"]& \mathscr{O}_{Z}\biggr\lvert_{\alpha}.
\end{tikzcd}
\end{center}
Hence $\zeta_1^{\#}=\zeta_2^{\#}$, and the inclusion map $i:U \to X$ is a monomorphism in the category of real closed spaces. 
\end{ex}

If $Z:=X \setminus U$ for $U$ open, $Z$ has the structure of a closed topological subspace, and $\ox|_{Z}$ gives $Z$ a locally ringed structure. In the above example, $Z$ is a point, so it is not hard to see $Z$ is closed under specialization, but this holds for any $Z$ which is the complement of an open. Together, we have

\begin{defn}
Let $X$ be a real closed space, and let $Z$ be a topological subspace of $X$ which is closed in $X$. Then the space $Z$ together with the structure sheaf $\mathscr{O}_{Z}:=\ox|_{U}$ is a \textit{closed subspace} of $X$.
\end{defn}

From topology, we know taking the intersection of infinitely many open sets may not be open, but in the real closed category, such intersections are proconstructible. Let's see an example;

\begin{ex}
Let $X=\sper \R[x,y]$. For each natural $n$, let $\Phi_{n}$ in $\lorf \R[x,y]$ be the sentence
\[ \exists X,Y \text{ s.t. } X^{2}+Y^{2}<\frac{4}{n}.  \]
Each $\Phi_{n}$ defines an open subspace $U_{n} \subset X$, and let 
\[ W=\bigcap_{n \in \N} U_{n}.  \]
We certainly have the homomorphism $\R[x,y] \to \R$ given by $x,y \mapsto 0$ satisfies $\Phi_{n}$ for all natural $n$, so $W$ contains at least one closed point. Taking the transcendence degree 2 extension $\R(x,y)$ of $\R$, we can take the real closure of $\R(x,y)$ with respect to any ordering. For example, the set
\[   \gamma:=\left\{ f(x,y) \in \R(x,y): \exists \epsilon>0 \text{ such that } f\biggr\lvert_{[0,2\epsilon] \times [0,\epsilon]} \geq 0 \right\} \]
orders $\R(x,y)$, and $\R[x,y] \to \R(\gamma)$ satisfies $\Phi_n$ with $2Y=X=\epsilon$, and so $\gamma \in W$.
\end{ex}

In fact, points of $W$ are exactly the generizations of $\R[x,y] \to \R$ with $x,y \mapsto 0$. Let $\lambda \subset \R[x,y]$ be the fixed closed point of $W$, where $\R(\lambda) \cong \R$ and $\text{ht}(\supp(\lambda))=2$. If $\alpha \subseteq \lambda$ is a point of $\sper \R[x,y]$, $\supp(\alpha)$ is a convex prime ideal of height 0 or 1. If $\supp(\alpha)$ has height 0, $\supp(\alpha)=\langle 0 \rangle$, and $\alpha$ is some positive cone which orders $\R(x,y)$ as in Example 54 (\cite{rag},157). If $\text{ht}(\supp(\alpha))=1$, then 
\[ \R[x,y] \to \R[x,y]/\supp(\alpha) \to \R(\alpha) \cong  \R(t), \]
where $\R(t)$ is a hyperreal extension of $\R$. Since $\R$ is complete, there are exactly two hyperreal extensions of $\R$ for each real number $r \in \R$ (\cite{rag},254). 

\begin{defn}
(\cite{rcs},46) If $X$ is a real closed space with $\lambda \in X$, The collection of generizations $W_{\lambda}$ of $\lambda$ is called the \textit{local subspace of} $X$ \textit{at} $\lambda$. 
\end{defn}

Since local subspaces $W_{\lambda} \subseteq X$ are formed by taking infinite intersections, there is not necessarily single sentence defining the local subspace, and so local subspaces are not usually affine real closed, nor are there open affine subspaces of local subspaces given by open affines in $X$.

Ideally, we would like to have something analogous to a local description of local subspaces, and we obtain such a description by examining the maps from valuative spaces to local subspaces. '

\begin{defn}
(\cite{rcs},44) Let $\ell$ be an orderable field with $A \subseteq \ell$ a valuation ring. In a real closure $\rho \ell$ of $\ell$ with respect to an ordering $\alpha$, take the intersection $\bigcap C$ of all convex subrings $C$ of $\rho \ell$ containing $A$. Denoting this intersection by $C_{\alpha}$, $\sper C_{\alpha}$ has at least two points; $\beta$, which corresponds to the maximal ideal $m_{\alpha} \subset C_{\alpha}$, and $\alpha$  corresponding to the ordering of $\ell$. We call \[ V=\{\alpha, \beta\} \subseteq \sper C_{\alpha} \] a \textit{valuative subspace} of $\sper C_{\alpha}$.
\end{defn}

We can now use valuative spaces to probe the structure of local spaces. Let $X$ be a real closed space with $x \in X$, and let $W_{x}$ be the local subspace of $X$ at $x$. Let $w \neq x$ with $w$ specializing to $x$, and $\alpha_x$ and $\alpha_w$ are orderings corresponding to $x$ and $w$ in an open affine neighborhood of $x$. We have a diagram
\begin{center}
\begin{tikzcd}
\sper \rho(\alpha_x) \ar[r] & W_x \ar[d,hook] \\ & X. 
\end{tikzcd}
\end{center}
There is a convex valuation ring \[ C_{xw} \subseteq \rho(\alpha_x) \subseteq \rho(\alpha_w), \] and let $V \subseteq \sper C_{xw}$ be the valuative space consisting of $x$ and $w$. We can then extend our diagram to a square;
\begin{center}
\begin{tikzcd}
\sper \rho(\alpha_x) \ar[r] \ar[d,"x" swap] & W_x \ar[d,hook] \\ V \ar[r,hook] & X, 
\end{tikzcd}
\end{center}
and since $w \in W_x$, we have a lift 
\begin{center}
\begin{tikzcd}
\sper \rho(\alpha_x) \ar[r] \ar[d,"x" swap] & W_x \ar[d,hook] \\ V \ar[r,hook] \ar[ru, dotted] & X. 
\end{tikzcd}
\end{center}
So our choice of generalization of $x$ gives a unique lift in our commutative diagram.

We can generalize this concept substantially. Let $f:X \to Y$ be any morphism of real closed spaces with $x$ in the fiber over $y$ and $y'$ a specialization of $y$.  Let $\sper \rho(\alpha_x) \to X$ be the morphism specifying $x$. Let $C_{y'y}$ be the convex valuation ring of the specialization $y' \in \overline{\{y\}}$, and let $V \subseteq \sper C_{y'y}$ be the valuative space of these points. We then have a commutative diagram
\begin{center}
\begin{tikzcd}
\sper \rho(\alpha_x) \ar[r] \ar[d,"y" swap] & X \ar[d] \\ V \ar[r,hook] & Y, 
\end{tikzcd}
\end{center}
and this diagram can be constructed from any morphism of nonempty real closed spaces, and our choice of the specialization $y'$ of $y$ uniquely determines the space $V$. We urge the reader to revisit the diagram above when we discuss quasiseparated morphisms!

\begin{ex}
In $\sper \R_{0}[x,y,z]$ let the subspace $X$ be given by the equations $z=y^{2}$ and $y \neq 0$. Let $Y=\sper \R_{0}[u,v]$, and consider the morphism $f:X \to Y$ given on the level of rings by sending $u \mapsto x$ and $v \mapsto z$. Let $\sper \R_{0}(1,2) \to Y$ be the point corresponding to $(1,2)$ in $Y$. Since $(\sqrt{2})^{2}=2$ and $\sqrt{2} \in \R_{0}$, we know $(1,2)$ is in the set-theoretic image of $f$. Let 
\begin{center}
\begin{tikzcd}
\sper \R_{0}(\epsilon_{v}<2) \ar[r] \ar[d]  & Y \\ \sper \R_{0}(1,2) \ar[ru]  &
\end{tikzcd}
\end{center}
be a generization of $(1,2)$, and let $W_{2,\epsilon_{v}}$ be the corresponding valuative subspace. We see $(1,\sqrt{6},2)$ is in the fiber over $(1,2)$, so we have a commutative diagram
\begin{center}
\begin{tikzcd}
 \sper \R_{0}(1,\sqrt{6},2) \ar[r] \ar[d,"12" swap]  &  X \ar[d,"f"]  \\ W_{2,\epsilon_{v}} \ar[r, "j" hook]  &  Y.  
\end{tikzcd}
\end{center}
While we cannot always lift the morphism between the valuative space and $Y$ to a morphism between the valuative space and $X$, we can in this example. We know the set-theoretic image of $f$ is the subspace
\[  \{\R_{0}[u,v] \to \rho: v>_{\rho}0  \},  \]
and we may define a morphism $g:\im(f) \to X$ given by the homomorphism \[\Gamma(\sper \left(\R_{0}[x,y,z]/\langle z-y^{2} \rangle)_{y}\right) \to \Gamma(\sper \R_{0}[u,v])\] determined by $x \mapsto u,z \mapsto v$, and $y$ is mapped to the unique positive square root of $v$. Since $X$ may be written as the disjoint union of $X \cap (y>0)$ and $X \cap (y<0)$, we see $\im(g)=X \cap (y>0)$, and $f \circ g = \text{id}_{\im(f)}$. We then have a section by composing the local section $g$ with $j$;
\begin{center}
\begin{tikzcd}
 \sper \R_{0}(1,\sqrt{6},2) \ar[r] \ar[d,"12" swap]  &  X \ar[d,"f"]  \\ W_{2,\epsilon_{v}} \ar[r, "j", hook, swap] \ar[ru,"g \circ j", swap, dotted] &  Y.  
\end{tikzcd}
\end{center}
\end{ex}
 
\subsection{Morphisms}

The last example of the previous section begs the question, is the image of a morphism always a subspace? While the answer is no generally, the answer is yes if the morphism is quasicompact.

\begin{defn}
(\cite{rcs},56) A morphism $f:X \to Y$ of real closed spaces is \textit{quasicompact} if for all open affine $\sper B \subseteq Y$, $f^{-1}(\sper B)$ is a quasicompact subspace of X. More explicitly, any open cover $\{\sper A_{\lambda}\}_{\lambda \in \Lambda}$ of $f^{-1}(\sper B)$ has a finite subcover $\sper A_1, \hdots ,\sper A_n$.
\end{defn}

\begin{prop}
(\cite{rcs}) If $f:X \to Y$ is a quasicompact morphism of real closed spaces and $K$ is a constructible subspace of $X$, then $f(K)$ is locally constructible in $Y$.
\end{prop}
\begin{proof}
We want to prove, in any open $\sper B$ of $Y$ containing a point in the image , that $f(K) \cap \sper B$ is constructible. Since $f$ is quasicompact, we can write  
\[  f^{-1}(\sper B)= \bigcup_{i=1}^{n} \sper A_i  \]
and for a single open affine $\sper A_j$,
\begin{center}
\begin{tikzcd}
K \cap \sper A_j \ar[d,hook] \ar[rd] &  \\  \sper A_j \ar[r,"f\lvert_{\sper A_j}" swap]  &  \sper B  
\end{tikzcd}
\end{center}
commutes. By Theorem 18 \cite{coste_roy_Topologie_1982}, the image of $K \cap \sper A_j$ in $\sper B$ is constructible, and so the finite union 
\[ \bigcup_{i=1}^{n} f(K \cap \sper A_i)=f(K) \cap \sper B  \]
is constructible as well. 
\end{proof}

If $X=\sper A$ and $Y=\sper B$, for any basic open $D(g) \subseteq Y$, the limit over the diagram
\begin{center}
\begin{tikzcd}
  & \sper A \ar[d]  \\  D(g) \ar[r,hook] & \sper B
\end{tikzcd}
\end{center}
is 
\begin{center}
\begin{tikzcd}
 D\left(f^{\#}(g)\right) \ar[r,hook] \ar[d,"f\lvert_{D\left(f^{\#}(g)\right)}", swap] & \sper A \ar[d,"f"]  \\  D(g) \ar[r,hook] & \sper B,
\end{tikzcd}
\end{center}
so every basic open set has basic inverse image, which is itself open and affine. We have now proved

\begin{prop}
(\cite{rcs},58) Morphisms of affine real closed spaces are quasicompact.    
\end{prop}

Luckily, quasicompact morphisms are stable under base change, just as they are in the category of schemes:

\begin{prop}
(\cite{rcs},59) If $f:X \to Y$ is a quasicompact morphism, and $t:T \to Y$ is any morphism, we may form the fiber product 
\begin{center}
\begin{tikzcd}
 T \times_{Y} X \ar[r,"p_X"] \ar[d,"p_T",swap] &  X \ar[d,"f"] \\  T \ar[r,"t"] &  Y.  
\end{tikzcd}
\end{center}
If $f$ is quasicompact, then so is $p_T$.
\end{prop}
\begin{proof}
We will proceed to show $p_T$ is quasicompact by showing that $T$ has an open cover, and when we restrict the codomain to each affine in the cover, we obtain a quasicompact morphism. Choose an open affine $\sper B$ of $Y$, and base change to obtain the fiber square
\begin{center}
\begin{tikzcd}
 p_{X}^{-1}(f^{-1}(\sper B)) \ar[r] \ar[d]   &  f^{-1}(\sper B) \ar[d]  \\  t^{-1}(\sper B) \ar[r]  &   \sper B.
\end{tikzcd}
\end{center}
since $f$ is quasicompact, we rewrite
\[  f^{-1}(\sper B)=\bigcup_{i=1}^{n} \sper A_i , \]
which allows us to rewrite our diagram as
\begin{center}
\begin{tikzcd}
 p_{X}^{-1}\left(\bigcup_{i=1}^{n} \sper A_i \right) \ar[r] \ar[d]   &  \bigcup_{i=1}^{n} \sper A_i \ar[d]  \\  t^{-1}(\sper B) \ar[r]  &   \sper B.
\end{tikzcd}
\end{center}
Choose an open affine cover $\{\sper R_{\lambda}\}_{\lambda \in \Lambda}$ of $t^{-1}(\sper B)$. For a fixed $\lambda$, we may take the fiber product over 
\begin{center}
\begin{tikzcd}
&  p_{X}^{-1}(\bigcup_{i=1}^{n} \sper A_i) \ar[r] \ar[d]   &  \bigcup_{i=1}^{n} \sper A_i \ar[d]  \\ \sper R_{\lambda} \ar[r] &  t^{-1}(\sper B) \ar[r]  &   \sper B,
\end{tikzcd}
\end{center}
which is
\begin{center}
\begin{tikzcd}
\bigcup_{i=1}^{n} \sper( R_{\lambda} \otimes_{B} A_{i}) \ar[r] \ar[d] &  p_{X}^{-1}(\bigcup_{i=1}^{n} \sper A_i) \ar[r] \ar[d]   &  \bigcup_{i=1}^{n} \sper A_i \ar[d]  \\ \sper R_{\lambda} \ar[r] &  t^{-1}(\sper B) \ar[r]  &   \sper B.
\end{tikzcd}
\end{center}
Since $\sper B$ is arbitrary and each point of $T$ is contained in some $t^{-1}(\sper B)$, we have shown that 
\[  p_{X}^{-1}\left(\bigcup_{i=1}^{n} \sper A_i \right) \to t^{-1}(\sper B) \]
is quasicompact independent of $B$, implying that $p_{T}$ is quasicompact as well. 
\end{proof}

Quasiseparated morphisms are similar; 

\begin{defn}
Let $f:X \to Y$ be a morphism, and consider the fibered square
\begin{center}
\begin{tikzcd}
X \times_{Y} X \ar[r,"p_X"] \ar[d,"p_X",swap] & X \ar[d,"f" ] \\ X \ar[r,"f" swap] & Y.
\end{tikzcd}
\end{center}
The pair of identity morphisms to $X$ give a morphism
\[  \Delta_{f}:X \to X \times_{Y} X \]
called the diagonal morphism of $f$. We say $f$ is \textit{quasiseparated} is $\Delta_f$ is quasicompact.
\end{defn}

\begin{prop}
(\cite{rcs}, 59) Let 
\begin{center}
\begin{tikzcd}
X \ar[r,"f"]  &  Y \ar[r,"g"] &  Z 
\end{tikzcd}
\end{center}
be morphisms of real closed spaces such that $g \circ f$ is quasicompact and $g$ is quasiseparated. Then $f$ is quasiseparated.
\end{prop}
\begin{proof}
This proof is entirely an application of base change. We have a diagram
\begin{center}
\begin{tikzcd}
    &  Y \ar[d,"g"] \\ X \ar[r,"g \circ f"] &  Z  
\end{tikzcd}
\end{center}
with limit
\begin{center}
\begin{tikzcd}
  X \times_{Z} Y \ar[r,"p_{Y}"] \ar[d] &  Y \ar[d,"g"] \\ X \ar[r,"g \circ f"] &  Z.  
\end{tikzcd}
\end{center}
Since quasicompactness is stable under base change, $p_Y$ is quasicompact. We also have a diagram
\begin{center}
\begin{tikzcd}
 & X \times_{Z} Y \ar[d,"f\times \text{id}_{Y}"]   \\  Y \ar[r,"\Delta_g"]     &   Y \times_{Z} Y   
\end{tikzcd}
\end{center}
with limit 
\begin{center}
\begin{tikzcd}
 X \times_{Y} Y  \ar[r,"p_2"] \ar[d] & X \times_{Z} Y \ar[d,"f\times \text{id}_{Y}"]   \\  Y \ar[r,"\Delta_g"]     &   Y \times_{Z} Y.   
\end{tikzcd}
\end{center}
By the properties of fiber products, $X \times_{Y} Y \cong X$, and the induced morphism from $X$ to $X \times_{Z} Y$ by composing $p_2$ with this isomorphism is the graph $\Gamma_f$ of $f$. Since $\Delta_g$ is quasicompact by assumption, $\Gamma_f$ is quasicompact by base change. Since compositions of quasicompact morphisms are quasicompact,
\begin{center}
\begin{tikzcd}
X \ar[r,"\Gamma_f"] & X \times_{Z} Y \ar[r,"p_{Y}"] & Y  
\end{tikzcd}
\end{center}
is quasicompact and equals $f$, so the claim is proved.
\end{proof}

Quasiseparated morphisms, just like quasicompact morphisms, are closed under composition and base change.

\begin{prop}
(\cite{rcs},61) Quasiseparated morphisms are closed under base change; that is, if $f:X \to Y$ is quasiseparated and $T \to Y$ is any morphism, then $f_{T}:X_{T} \to T$ is quasiseparated, where 
$X_T=X \times_{Y} T$.
\end{prop}
\begin{proof}
We aim to show that the diagonal $X_{T} \to X_{T} \times_{T} X_{T}$ is quasicompact. First, by associativity of fiber products,
\[ X_T \times_{T} X_T =\left(X \times_{Y} T\right) \times_{T} \left(X \times_{Y} T \right) \cong X \times_{Y} \left(T \times_{T} X \right) \times_{Y} T \cong X \times_{Y} X \times_{Y} T. \]
Consider the fiber product 
\begin{center}
\begin{tikzcd}
X \times_{Y} X \times_{Y} T \ar[r] \ar[d]   &  T \ar[d]  \\    X \times_{Y} X \ar[r]  &     Y,     
\end{tikzcd}
\end{center}
which is isomorphic to 
\begin{center}
\begin{tikzcd}
X_{T} \times_{T} X_{T}  \ar[r] \ar[d]   &  T \ar[d]  \\    X \times_{Y} X \ar[r]  &     Y     
\end{tikzcd}
\end{center}
via the unique diagram isomorphism from associativity. If we then map into this diagram using the diagonal
\begin{center}
\begin{tikzcd}
&  X_{T} \times_{T} X_{T}  \ar[r] \ar[d]   &  T \ar[d]  \\   X \ar[r,"\Delta_{f}"] &   X \times_{Y} X \ar[r]  &     Y,    
\end{tikzcd}
\end{center}
we see the fiber product is 
\begin{center}
\begin{tikzcd}
X \times_{X} \left(X \times_{Y} T \right) \ar[r] \ar[d]  &  X_{T} \times_{T} X_{T}  \ar[r] \ar[d]   &  T \ar[d]  \\   X \ar[r,"\Delta_{f}"] &   X \times_{Y} X \ar[r]  &     Y   
\end{tikzcd}
\end{center}
by the magic square diagram, which is uniquely isomorphic to 
\begin{center}
\begin{tikzcd}
X_{T} \ar[r] \ar[d] &  X_{T} \times_{T} X_{T}  \ar[r] \ar[d]   &  T \ar[d]  \\   X \ar[r,"\Delta_{f}"] &   X \times_{Y} X \ar[r]  &     Y .  
\end{tikzcd}
\end{center}
Since $X_{T} \to X_{T} \times_{T} X_{T}$ is the base change of a quasicompact morphism, it is quasicompact, so $f_{T}$ is quasiseparated. 
\end{proof}

\begin{prop}
(\cite{rcs},61) In the composition 
\begin{center}
\begin{tikzcd}
X \ar[r,"f"] & Y \ar[r,"g"] & Z,
\end{tikzcd}
\end{center}
if $f$ and $g$ are quasiseparated, then so is $g \circ f$.
\end{prop}
\begin{proof}
We aim to show $\Delta_{g \circ f}:X \to X \times_{Z} X$ is quasicompact. By the universal property of fiber products, we have a morphism
\[  X \times_{Z} X \to Y \times_{Z} Y \]
induced by $f \times f$. We then have a diagram
\begin{center}
\begin{tikzcd}
   X \times_{Z} X \ar[r,"f \times f"]   &  Y \times_{Z} Y  & Y \ar[l,"\Delta_g"]
\end{tikzcd}
\end{center}
and the fiber product is
\begin{center}
\begin{tikzcd}
 X \times_{Y} X \ar[r] \ar[d] &   Y \ar[d,"\Delta_g"] \\   X \times_{Z} X \ar[r,"f \times f"]   &  Y \times_{Z} Y.  
\end{tikzcd}
\end{center}
The morphism $X \times_{Y} X \to X \times_{Z} X$ is the base change of $\Delta_g$, so it is quasicompact. We map $X$ to this square using
\begin{center}
\begin{tikzcd}
X \ar[rrd,"f"] \ar[rdd,"\Delta_{g \circ f}", swap] & & \\ & X \times_{Y} X \ar[r] \ar[d] &   Y \ar[d,"\Delta_g"] \\  & X \times_{Z} X \ar[r,"f \times f"]   &  Y \times_{Z} Y,  
\end{tikzcd}
\end{center}
and the unique induced morphism is in fact $\Delta_f$
\begin{center}
\begin{tikzcd}
X \ar[rd,dotted] \ar[rrd,"f"] \ar[rdd,"\Delta_{g \circ f}", swap] & & \\ & X \times_{Y} X \ar[r] \ar[d] &   Y \ar[d,"\Delta_g"] \\  & X \times_{Z} X \ar[r,"f \times f"]   &  Y \times_{Z} Y.  
\end{tikzcd}
\end{center}
We have now demonstrated $\Delta_{g \circ f}=(f \times f) \circ \Delta_{f}$, and since a composition of quasicompact morphisms is quasicompact, we are done.
\end{proof}

While the proofs are slightly different in some cases, pages 58-64 of \cite{rcs} give theorems on the intersections of quasicompact open sets in quasiseparated real closed spaces, in addition to some generalizations of theorems which hold for schemes to locally proconstructible spaces. All of which is to say, while the topology of real closed spaces is quite different from (and much finer than) that of schemes, the categorical properties are quite similar. Since quasiseparated, separated, quasicompact, and finite type morphisms are all morphism properties given in terms of diagrams, they look very similar in the scheme and real closed category. The situation is in sharp contrast with the definition of a regular morphism (\cite{rcs},67), which is unrelated to the notion of a regular local ring.

\subsection{A Deep Example: Semialgebraic Geometry \& Real Closed Spaces}

Some of the best behaved real closed spaces come from familiar classical geometric objects. 

Let $k$ be a real closed field, and let $\Phi(t_{1},\hdots,t_{n})$ be a sentence in $n$ variables in the language $\lorf(k)$. We may consider the sentence $\Phi$ in the language $\lorf(k[t_1 , \hdots,t_n])$, and by Section 1.5, there is an associated real closed space $K_{\Phi} \hookrightarrow \sper k[t_1 , \hdots,t_n]$. The closed points of $K_{\Phi}$ correspond to $k$-algebra homomorphisms $k[t_1 , \hdots,t_n] \to k$, or maximal convex ideals $\textbf{p} \in \text{Spec } k[t_1,\hdots,t_n]$. All maximal convex ideals of $k[t_1,\hdots,t_n]$ are ideals of the form $\langle t_1-a_1, \hdots, t_n-a_n \rangle$ for $(a_1,\hdots,a_n) \in k^{n}$. Hence the closed points of $K_\Phi$ correspond to the points in the real vector space $k^{n}$.

We can say more about real closed spaces $K_{\Phi} \hookrightarrow \sper k[t_1 , \hdots,t_n]$. Suppose $g(t_1,\hdots,t_n)>0$ is an atomic formula in $\Phi$, written in the language $\lorf(k[t_1 , \hdots,t_n])$. For any closed point $x \in K_{\Phi}$, the kernel $\textbf{p}$ of $x:k[t_1,\hdots,t_n] \to k$ is given by a vector $(a_1,\hdots,a_n) \in k^{n}$. By definition, 
\[ x \models g(t_1,\hdots,t_n)>0,  \]
and rephrasing for ease of understanding,
\[ g(x(t_1),\hdots,x(t_n))>0  \]
with $x(t_i) \in k$. We then know $x(t_i - a_i)=x(t_i)-a_i=0$, so 
\[ g(a_1,\hdots,a_n)>0,  \]
and each atomic formula gives a polynomial (in)equality describing a region of $k^n$. Since the atomic formula here can be replaced with any other atomic formula in $\Phi$, the real closed space $K_{\Phi}$ corresponds to the semialgebraic space defined by $\Phi$ in $k^{n}$.

Let's refer to real closed spaces like $K_{\Phi}$ as \textit{semialgebraic spaces}. Semialgebaic spaces are quasicompact and quasiseparated, and it is not difficult to prove. Perhaps the most important part about these spaces is that the ambiguity present in deciding which real closed field a subspace should be defined over evaporates. 

If $X$ is a real closed space, a curve in $X$ should be a 'one dimensional' subspace $C \hookrightarrow X$, but since $X$ may not be equidimensional, defining a curve this way is annoying. Instead, we prefer to think of a curve as a morphism
\[ \gamma:[0,1] \to X, \]
but what exactly do we mean by $[0,1]$? We may define $[0,1]$ over $\R_{0}$ using the sentence 
\[  \exists t:0 \leq t \leq 1  \]
in the language $\lorf(\R_{0}[t])$. Since we may replace $\R_{0}$ with any ring here, which ring should we choose to define our interval over?

With semialgebraic spaces, the choice is obvious. Suppose $K_{\Phi} \hookrightarrow \sper k[t_1 , \hdots,t_n]$ is a semialgebraic space. Define $[0,1]_{k}$ using the sentence 
\[  \exists t:0 \leq t \leq 1  \]
in the language $\lorf(k[t])$. If $C \subset K_{\Phi}$ is a subspace of dimension one, fix a closed point $x \in C \subset K_{\Phi}$, then let $W_{x}^{C}$ be the local subspace of $x$ in the curve $C$. We then have a commutative triangle describing $C$:
\begin{center}
\begin{tikzcd}
 \displaystyle [0,1]_{k}  \ar[rd] \ar[rr,"\gamma"] & & K_{\Phi} \ar[ld] \\ & \sper(k), &
\end{tikzcd}
\end{center}
where $\gamma$ factors through $C$. If the closed point $x$ is in the image of $\gamma$, there is a number $z \in k$ between 0 and 1 such that 
\[ \gamma\lvert_{W_z}:W_{z} \to W_{x}^{C} \]
is an isomorphism. Such a morphism $\gamma$ is like a local parameterization of $C$ in a neighborhood of $x$.

\subsection{The Real \'{E}tale Site}

Now, to form a site from the category of real closed spaces, we need some notion of \'{e}tale morphism for real closed spaces. To begin, we introduce finite type morphisms in the real closed category. This choice allows us strong control over the images of quasicompact morphisms of real closed spaces, while also making the model theory of real closed spaces an effective tool in translating between algebra (i.e. sheaves) and geometry of these spaces.

\begin{defn}
A morphism $f:X \to Y$ is \textit{locally of finite type} if for each point $x \in X$, there are open affine subspaces $V$ containing $f(x)$ and $U$ containing $x$ such that
\begin{center}
\begin{tikzcd}
U \ar[r,"f\lvert_{U}"] \ar[rd, "i_{U,V}" hook, swap] & V \\ & V \times \sper \R_{0}[t_1,\hdots,t_n] \ar[u,"\pi_{V}" swap]
\end{tikzcd}
\end{center}
commutes, where $\pi_V$ is projection from the fiber product over $\sper \R_{0}$, and $i_{U,V}$ is a monomorphism. $f$ is \textit{finite type} if it is locally of finite type and quasicompact.
\end{defn}

\begin{remark}
If $U \cong \sper A$ and $V \cong \sper B$ are affine, assume $A$ and $B$ are real closed. Since $B$ is an algebra over $\R_{0}$, we have 
\begin{align*}
V \times \sper \R_{0}[t_1,\hdots,t_n] = & \sper B \times \sper \R_{0}[t_1,\hdots,t_n] \\ \cong & \sper B \otimes_{\R_{0}} \R_{0}[t_1,\hdots,t_n] \\ \cong & \sper B[t_1,\hdots,t_n].
\end{align*}
Although the ring $B[t_1,\hdots,t_n]$ is not real closed, it is convenient to write its real spectrum as $\sper B[t_1,\hdots,t_n]$ to define $U$ as a constuctible subset, since we'd like to use the variables $t_i$ to write the sentence $\Phi_{U}$ which defines $U$ in $V \times \sper \R_{0}[t_1,\hdots,t_n]$.
\end{remark}

\begin{remark}
If $f:X \to Y$ is finite type, consider the factorization $X \to f(X) \to Y$ of $f$. Let $V$ be a nonempty open in $Y$, and choose open affine $U$ such that we have a commutative triangle 
\begin{center}
\begin{tikzcd}
U \ar[r,"f\lvert_{U}"] \ar[rd, hook] & V \\ & V \times \sper \R_{0}[t_1,\hdots,t_n]. \ar[u]    
\end{tikzcd}
\end{center}
We then have a commutative diagram
\begin{center}
\begin{tikzcd}
U \ar[r] \ar[rrd] & f(X) \ar[r,hook] & Y \\ & & V, \ar[u,hook]  
\end{tikzcd}
\end{center}
which yields a morphism 
\begin{center}
\begin{tikzcd}
U \ar[r] \ar[rd,dotted]  & f(X) \ar[r,hook] & Y \\ & f(X) \times_{Y} V \ar[u,hook] \ar[r,hook] & V. \ar[u,hook]  
\end{tikzcd}
\end{center}
We consequently know $U$ maps to $V$ in two ways;
\begin{center}
\begin{tikzcd}
U \ar[rr] \ar[rd,hook] &  &  f(X)\times_{Y}V \ar[d]  \\ & V \times \sper \R_{0}[t_1,\hdots,t_n] \ar[r]  & V. 
\end{tikzcd}
\end{center}
We may then map to the fiber product 
\begin{center}
\begin{tikzcd}
U \ar[r,dotted] \ar[rd,hook] & f(X)\times_{Y}V \times  \sper \R_{0}[t_1,\hdots,t_n] \ar[r] \ar[d]  &  f(X)\times_{Y} V \ar[d]  \\ & V \times \sper \R_{0}[t_1,\hdots,t_n] \ar[r]  & V, 
\end{tikzcd}
\end{center}
and since the morphism $U \to V \times \sper \R_{0}[t_1,\hdots,t_n]$ is a monomorphism, $U \to f(X)\times_{Y}V \times  \sper \R_{0}[t_1,\hdots,t_n]$ is as well. We know $f(X)\times_{Y}V \cong f(X) \cap V$, so in conclusion 
\begin{center}
\begin{tikzcd}
U \ar[r] \ar[rd,hook] & f(X) \cap V \\ & (f(X) \cap V)\times \sper \R_{0}[t_1,\hdots,t_n] \ar[u,"\pi_{V}\lvert_{f(X)}" swap]
\end{tikzcd}
\end{center}
commutes, and $X \to f(X)$ is a finite type surjection. We may thus check whether a morphism is finite type by checking if it is finite type onto its image!
\end{remark}

The previous remark essentially proves that if $X \to Y \to Z$ is locally of finite type, so is $X \to Y$
\begin{prop}
(\cite{rcs},90-91) If the composition $g \circ f$ of the morphisms $f:X \to Y$ and $g:Y \to Z$ is locally of finite type, then so is $f$    
\end{prop}
\begin{proof}
Analogously to the argument above, an open $U$ of $X$ maps to $Y \times_{Z} V$ for some open $V \subseteq Z$, and this factors through $Y \times_{Z} V \times \sper \R_{0}[t_1,\hdots,t_n]$ via a monomorphism followed by a projection.
\end{proof}

\begin{prop}
(\cite{rcs},90-91) If $f:X \to Y$  and $g:Y \to Z$ are locally of finite type, then so is $g \circ f$.
\end{prop}
\begin{proof}
If $g$ is locally finite type, there is a natural $m \in \N$ and an open $V \subseteq Y$ mapping to a nonempty open $W \subseteq Z$ such that 
\begin{center}
\begin{tikzcd}
V \ar[r,"g\lvert_{V}"] \ar[rd,hook] & W \\ & W \times \sper \R_{0}[t_1,\hdots,t_m] \ar[u]
\end{tikzcd}
\end{center}
commutes. If $f$ is locally of finite type, there is a natural $n$, an open $U \subseteq X$ and open $V \subseteq Y$ (shrink $V$ here if necessary) such that 
\begin{center}
\begin{tikzcd}
U \ar[r,"f\lvert_{U}"] \ar[rd,hook] & V \\ & V \times \sper \R_{0}[t_1,\hdots,t_n] \ar[u]
\end{tikzcd}
\end{center}
commutes. Taking the product of $V \hookrightarrow W \times \sper \R_{0}[t_1,\hdots,t_m]$ with $\sper \R_{0}[t_1,\hdots,t_n]$ yields a diagram
\begin{center}
\begin{tikzcd}
V \ar[r,hook] & W \times \sper \R_{0}[t_1,\hdots,t_m] \\ V \times \sper \R_{0}[t_1,\hdots,t_n] \ar[r,hook] \ar[u,"\pi_V"] & W \times \sper \R_{0}[t_1,\hdots,t_{n+m}]. \ar[u]
\end{tikzcd}
\end{center}
We consequently have a commutative triangle
\begin{center}
\begin{tikzcd}
U \ar[rd,hook] \ar[r,"(g \circ f)\lvert_{U}"] & W \\ & W \times \sper \R_{0}[t_1,\hdots,t_{n+m}], \ar[u,"\pi_{W}"swap]
\end{tikzcd}
\end{center}
concluding this proof.
\end{proof}

\begin{prop}
(\cite{rcs},90-91) If $f:X \to Y$ is locally of finite type and $t:T \to Y$ is an arbitrary morphism, $f_T:X_{T} \to T$ is locally of finite type.
\end{prop}
\begin{proof}
The idea here is since 'locally of finite type' is a property given in terms of commutative diagrams, base changing the diagram 
\begin{center}
\begin{tikzcd}
& U \ar[ld,hook] \ar[r,hook] \ar[d] & X \ar[d,"f"] \\ V \times \sper \R_{0}[t_1,\hdots,t_n] \ar[r,"\pi_V"]  & V \ar[r,hook] &  Y
\end{tikzcd}
\end{center}
over $Y$ by the morphism $t$ gives the desired commutative triangle 
\begin{center}
\begin{tikzcd}
 \pi_{X}^{-1}(U) \ar[r] \ar[rd,hook] & t^{-1}(V)  \\  & t^{-1}(V) \times \sper \R_{0}[t_1,\hdots,t_n]. \ar[u,"\pi_{t^{-1}(V)}" swap] 
\end{tikzcd}
\end{center}
I leave it to the reader to check the relevant fiber product isomorphisms.
\end{proof}

\begin{defn}
A morphism $f:X \to Y$ of real closed spaces is \textit{real \'{e}tale} if $f$ is a quasicompact and quasiseparated morphism, and a local isomorphism which is locally of finite type.
\end{defn}

As with the usual definition of \'{e}tale, we need real \'{e}tale morphisms to be closed under composition and base change, as well as have left cancellation.

\begin{prop}
Let $f:X \to Y$, $g:Y \to Z$ and $t:T \to Y$ be morphisms of real closed spaces. Then 
\begin{enumerate}
\item If $f \, \text{and } \, g$ are real \'{e}tale, then so is $g \circ f$
\item If $f$ is real \'{e}tale and $t$ is arbitrary, then the induced morphism $X \times_{Y} T \to T$ is real \'{e}tale
\item If $g \circ f$ and $g$ are both real \'{e}tale, then so is $f$.
\end{enumerate}
\end{prop}
\begin{proof}
Since quasicompact morphisms, quasiseparated morphisms, locally finite type morphisms and local isomorphisms are closed under composition, real \'{e}tale morphisms are closed under composition as well. In fact, we've shown morphisms which are locally of finite type have all three properties in the above arguments, so we only need to prove quasicompact, quasiseparated, local isomorphisms have the listed properties.
\par To prove real \'{e}tale morphisms are closed under base change, we note that quasiseparated and quasicompact morphisms were shown previously to be closed under base change, and if we can show local isomorphisms are closed under base change, we are done. Fix a pullback square 
\begin{center}
\begin{tikzcd}
 T \times_{Y} X \ar[r,"\pi_X"] \ar[d,"\pi_T"] & X \ar[d,"f"] \\ T \ar[r,"t"] &  Y
\end{tikzcd}
\end{center}
with $f$ real \'{e}tale, and let $q$ be a point in $T \times_{Y} X$. Fix an open neighborhood $U$ about $\pi_{X}(q)$ on which $f|_{U}$ is an isomorphism onto an open $V \subseteq Y$. We make take the fiber product of the above diagram with $V$ over $Y$ to yield
\begin{center}
\begin{tikzcd}
 t^{-1}(V) \times_{V} f^{-1}(V) \ar[r,"\pi_X"] \ar[d,"\pi_T"] & f^{-1}(V) \ar[d,"f"] \\ t^{-1}(V) \ar[r,"t"] &  V.
\end{tikzcd}
\end{center}
We know $U \subseteq f^{-1}(V)$, and that the restriction of $f$ to $U$ is an isomorphism. Consequently, in the larger fibered square
\begin{center}
\begin{tikzcd}
U \times_{V} t^{-1}(V) \ar[d,"\pi_{U}" swap] \ar[r,"\pi_{t}"] & t^{-1}(V) \times_{V} f^{-1}(V) \ar[d,"\pi_X"] \ar[r,"\pi_T"] & t^{-1}(V) \ar[d,"t"] \\ U \ar[r,hook] & f^{-1}(V) \ar[r,"f"] &  V,
\end{tikzcd}
\end{center}
the morphism $U \times_{V} t^{-1}(V) \to t^{-1}(V)$ is an isomorphism. We have now demonstrated the existence of a neighborhood $U \times_{V} t^{-1}(V)$ of $q$ and an isomorphism onto its image $t^{-1}(V)$.
\par Finally, we prove left cancellation holds for real \'{e}tale morphisms. Suppose $g \circ f$ and $g$ are real \'{e}tale morphisms. Since real \'{e}tale morphisms are quasiseparated, and quasiseparated morphisms have left cancellation, $f$ is quasiseparated, and since $g \circ f$ is quasicompact as well, $f$ is quasiseparated. Finally, to show $f$ is a local isomorphism, let $x \in X$, and let $U_0$ be a neighborhood of $x$ on which $(g \circ f)\lvert_{U_{0}}$ is an isomorphism onto a neighborhood $U_2$ of $g(f(x))$. Since $g$ is a local isomorphism, select a neighborhood $U_1 \subseteq g^{-1}(U_{2})$ of $f(x)$ such that $g\lvert_{U_{1}}$ is an isomorphism onto an open $V_{2} \subseteq U_{2}$. We then have an isomorphism
\begin{align*}
& \left(g\lvert_{U_1}\right)^{-1}\circ (g \circ f)\biggr\lvert_{(g \circ f)^{-1}(V_{2})} \\ =& \left(g\lvert_{U_1}\right)^{-1}\circ g\lvert_{U_1} \circ f\biggr\lvert_{(g \circ f)^{-1}(V_{2})}  \\ =& f\biggr\lvert_{(g \circ f)^{-1}(V_{2})},
\end{align*}
so $f$ is real \'{e}tale and the proposition is proved
\end{proof}

\begin{remark}
The following defintion is to be sharply contrasted with Scheiderer's real surjective families and his construction of the real \'{e}tale site found in \cite{Scheiderer_94}. While Scheiderer constructs the real \'{e}tale site over a scheme $X$, he completely avoids defining 'real \'{e}tale' in the real closed category. This choice allows Scheiderer to use the full power of known results in the scheme category, but Scheiderer's approach does not take advantage of how structured real closed spaces are as a category. Additionally, since quasicompact morphisms of real closed spaces have constructible images, we can certainly define site and topos structures on images of real \'{e}tale morphisms. These subsets and their intersections are the main subsets of real closed spaces we compute with, and there is no need to give such structures to arbitrary subsets as was done in (\cite{Scheiderer_94},8). I will try to point out other moments of sharp contrast to Scheiderer's work as they arise. 
\end{remark}

In order to put a site structure suitable for motivic cohomology on the category of real closed spaces over a fixed real closed $S$, we need to specify a definition of Nisnevich cover. Luckily, since our definition of real \'{e}tale morphisms are local isomorphisms, real \'{e}tale morphisms easily lend themselves to defining Nisnevich covers. Also, since real \'{e}tale morphisms are quasicompact, every cover can be refined to a finite cover. We make the following precise below: 

\begin{defn}
A \textit{real \'{e}tale cover} of a real closed space $X$ is a collection of real \'{e}tale morphims $\{p_{\lambda}:U_{\lambda} \to X\}_{\lambda \in \Lambda}$ such that 
\[  \bigcup_{\lambda \in \Lambda} p_{\lambda}(U_{\lambda})=X.  \]
\end{defn}

\begin{prop}
Let $\{p_{\lambda}:U_{\lambda} \to X\}_{\lambda \in \Lambda}$ be a real \'{e}tale cover of a quasicompact real closed space $X$. Then
\begin{enumerate}[i.]
\item we can refine the cover $\{p_{\lambda}\}_{\lambda \in \Lambda}$ to be finite; \[ \{p_{i}:U_{i} \to X\}_{i=1}^{n} \subseteq \{p_{\lambda}:U_{\lambda} \to X\}_{\lambda \in \Lambda}  \]
and 
\item in this finite refinement $\{p_{i}:U_{i} \to X\}_{i=1}^{n}$, for each $x \in X$, there are points $y_{i_1} \in U_{i_1}, \hdots, y_{i_k} \in U_{i_k}$ such that $\rho(x) \to \rho(y_{i_j})$ is an isormorphism of real closed fields for each $j$.
\end{enumerate}
\end{prop}
\begin{proof}
We can prove ii. by noting that for $j \in \{1,2,\hdots,k\}$, there is a neighborhood $W_{i_j}$ of $x$ which is mapped to isomorphically by some open $V_{i_j} \subseteq U_{i_j}$. We may then restrict the codomain of each $p_{i_j}$ to \[ \bigcap_{j=1}^{k} W_{i_j}, \]
and since $p_{i_j}$ is a local isomorphism, $y_{i_j}=p_{i_j}^{-1}(x)$ is our desired set of points.
\par To prove i., we simply note that local isomorphisms are open maps, and so we refine the cover $\{p_{\lambda}(U_{\lambda})\}$ to a finite cover $\{p_i(U_{i})\}_{i=1}^{n}$. The real closed space morphisms $\{p_i\}_{i=1}^{n}$ give a finite real \'{e}tale cover of $X$, concluding this argument. 
\end{proof}

\begin{remark}
Since each point of $X$ is in the image of a local isomorphism of real closed spaces from our real \'{e}tale cover, it follows that all real \'{e}tale covers have the Nisnevich property. 
\end{remark}

If we fix a real closed space $X$, we have a natural candidate for a site structure on $X$ given by the real \'{e}tale covers, and we formalize this in the proposition below;
\begin{prop}
Let $\rce /X$ be the category with objects as real closed spaces with real \'{e}tale morphisms to $X$, with morphisms given by commutative triangles over $X$. Then $\rce /X$ together with $\tau$, the real \'{e}tale covers of $X$, is a site $(\rce /X,\tau)$. 
\end{prop}
\begin{proof}
Since real \'{e}tale morphisms are stable under base change and composition, real \'{e}tale covers are as well, which proves $\tau$ satisfies the first three axioms of a Grothendieck pretopology. Since isomorphisms are quasicompact, quasiseparated, local isomorphisms, $\tau$ also satisfies the fourth axiom, and $(\rce /X,\tau)$ is a site.
\end{proof}

Most things we can do with topological spaces also carry over to sites, including sheaves!

\begin{defn}
A functor $\mathscr{F}:(\rce /X)^{\text{op}} \to \underline{Sets}$ is a \textit{presheaf} on the site $\rce /X$, which is a \textit{sheaf} if for any cover $\{p_{\lambda}:U_{\lambda} \to X\}_{\lambda \in \Lambda} \in \tau$, the sequence
\begin{center}
\begin{tikzcd}
  \mathscr{F}(X) \ar[r] & {\displaystyle \prod_{\lambda \in \Lambda}\mathscr{F}(U_{\lambda})} \ar[r,bend left=10] \ar[r,bend right=10] &  {\displaystyle \prod_{\kappa,\lambda \in \Lambda} \mathscr{F}(U_{\kappa} \times_{X} U_{\lambda}) }
\end{tikzcd}
\end{center}
is an equalizer sequence.     
\end{defn}

The sheaves on a Groethendieck site form a topos, and this will be handy for comparing the work of Scheiderer to this text. 

\subsection{Relationship between Scheiderer's $\text{Sh}(Et/X)$ and $\text{Sh}(\rce/\rho X)$}

We will use Lemma 7.29.1 of the stacks project to show these categories of sheaves are equivalent, so we state the lemma in full below

\begin{lemma}
(\cite{stacks-project}, 7.29.1) Let $u:\mathscr{C} \to \mathscr{D}$ be a functor between sites such that 
\begin{enumerate}[i.]
\item $u$ is cocontinuous,
\item $u$ is continuous,
\item For any morphisms $a,b:U' \to U$ in $\mathscr{C}$ such that $ua=ub$, there is a covering $\{f_i:U_{i}' \to U'\}_{i \in I}$ such that $a \circ f_i = b \circ f_i$ for all $i$,
\item For any morphism $c:u(U') \to u(U)$ in $\mathscr{D}$, there is a covering $\{f_i:U_{i}' \to U'\}_{i \in I}$ and morphisms $\{c_i:U_{i}' \to U\}_{i \in I}$ such that $uc_{i}=c \circ uf_{i}$, and
\item For any object $V$ of $\mathscr{D}$, there is a covering family $\{u(U_i) \to V\}_{i \in I}$ of $\mathscr{D}$.
\end{enumerate}
Then the induced morphism 
\[ g=(u^{p},u_{p}):\text{Sh}(\mathscr{C}) \to \text{Sh}(\mathscr{D})  \]
of topoi is an equivalence of categories.
\end{lemma} 

We will apply this to our sheaf topoi, where our functor $u$ is a version of the real closure operation

\begin{defn}
Suppose $X$ is a scheme with a real closed point $\text{Spec}(k) \to X$. On an open affine cover $\{\text{Spec}(A_{\ell})\}_{\ell \in L}$ of $X$, we have the composition of functors $\sper(\cdot) \circ \Gamma$ which sends each open affine to the real closed space $\sper A_{\ell}$. Since the real closure functor commutes with taking quotients and localization by section 1.7, the real spectra $\sper A_{\ell} \cong \sper \rho A_{\ell}$ glue to form a real closed space $\rho X$. We thus have a real closure functor 
\[  \rho:(Et/X,\text{ret}) \to (\rce/\rho X,\tau)   \]
from the category of etale morphisms over a fixed scheme $X$ to the category $\rce/\rho X$, where $\text{ret}$ specifies coverings as collections $f_\lambda$ of \'{e}tale maps of schemes with codomain $U$ such that the collection $\rho f_\lambda$ surjects onto $\rho U$. 
\end{defn} 

\begin{thm}
$\rho$ induces an equivalence of topoi between $\text{Sh}(Et/X)$ and $\text{Sh}(\rce/\rho X)$.
\end{thm}
\begin{proof}
We apply Lemma 80 to prove the claim, so we will proceed to show all five conditions of the lemma are satisfied. 
\par To prove $\rho$ is cocontinuous, let $\{ f_{j}:V_{j} \to \rho U\}_{j \in J}$ cover $\rho U$ for some scheme $U$. Let $\sper B_{i}$ be an affine open cover of $\rho U$, and choose $\sper A_{\lambda} \to f_{j}^{-1}(\sper B_i)$ a cover of $f_{j}^{-1}(\sper B_i)$ wherein each $\sper A_{\lambda}$ is is chosen small enough that $f_{j}\lvert_{\sper A_{\lambda}}$ is an isomorphism and $A_\lambda$ is real closed. For each triple $(i,j,\lambda)$ we have a commutative diagram
\begin{center}
\begin{tikzcd}
\sper A_{\lambda} \ar[r ,hook] \ar[d,hook] & f_{j}^{-1}(\sper B_{i}) \ar[r,hook] \ar[d,"f_{j}\lvert"] &  V_{j} \ar[d] \\ \sper B_i[t_1,\hdots,t_n] \ar[r] &  \sper B_{i} \ar[r,hook]  &  \rho U ,
\end{tikzcd}
\end{center}
where each horizontal hook arrow is an open immersion. Since $\lambda$ depends on $j$, the $\rce$ morphisms
\[ \sper A_{\lambda} \hookrightarrow V_{j} \to \rho U \]
give a refinement of the covering $\{ f_{j}:V_{j} \to \rho U\}_{j \in J}$. I further claim that $\{\text{Spec}(A_{\lambda}) \to U \}_{\lambda \in \Lambda}$ is a cover of $U$ in $\text{ret}$. It is slightly easier to show $\{\text{Spec}(A_{\lambda}) \to \text{Spec}(B_{i}) \}_{(i,\lambda) \in I \times \Lambda}$ is a cover, and this implies our claim because any cover of a $\text{ret}$ cover is a $\text{ret}$ cover. Cover $f_j(\sper A_{\lambda})$ with affine opens $\sper(B_{i,b_1}),\hdots ,\sper(B_{i,b_m})$, where $B_{i,b_k}$ is a localization of both $A_{\lambda}$ and $B_i$. We have a commutative diagram of rings
\begin{center}
\begin{tikzcd}
 B_i \ar[r] \ar[rd] & B_{i}[t_1,\hdots,t_n] \ar[r] \ar[d] & A_{\lambda}  \\  & B_{i,b_k}  &
\end{tikzcd}
\end{center}
for each localization $B_{i,b_k}$, so by the universal property of the equalizer sequence, we obtain a diagram 
\begin{center}
\begin{tikzcd}
A_{\lambda} \ar[r]  & {\displaystyle \prod_{k}}B_{i,b_k} \ar[r,bend left=10] \ar[r,bend right=10]   & {\displaystyle \prod_{k,l}} B_{i,b_{k}b_{l}}  \\  & B_i[t_1,\hdots,t_n]. \ar[u] \ar[lu,dotted,"\exists !" ] &  
\end{tikzcd}
\end{center}
We extend our open affine covering $\sper(B_{i,b_{k}})$ to an open affine covering $\sper(B_{i,b_{k'}})$ of $\sper B_i$, and this new cover yields an equalizer sequence of $B_i$-modules
\begin{center}
\begin{tikzcd}
 B_i \ar[r] & {\displaystyle \prod_{k}}B_{i,b_{k'}} \ar[r,bend left=10] \ar[r,bend right=10]  & {\displaystyle \prod_{k}}B_{i,b_{k'}b_{l'}} \\ A_{\lambda} \ar[r] \ar[u,dotted,"\exists !"]  & {\displaystyle \prod_{k}}B_{i,b_k} \ar[r,bend left=10] \ar[r,bend right=10] \ar[u] &  {\displaystyle \prod_{k}}B_{i,b_{k}b_{l}}. \ar[u]
\end{tikzcd}
\end{center}
Hence $A_{\lambda} \cong B_i$, proving $\text{Spec}(A_{\lambda}) \to U$ is an \'{e}tale morphism of schemes. 
\par The next argument will demonstrate $\rho$ is a continuous functor. Since tensor products of rings commute with real closures, we only need demonstrate that for any \'{e}tale cover $\{U_i \to U\}$ over $X$ in the ret topology, $\rho U_i \to \rho U$ is an $\rce$ cover over $\rho X$. For any \'{e}tale morphism $f:X \to Y$ , where $f$ is a morphism in a ret cover of $Y$, we have a commutative diagram
\begin{center}
\begin{tikzcd}
  & \text{Spec}(A) \ar[r,hook] \ar[d] \ar[ld,hook] & X \ar[d,"f"] \\ \text{Spec}(B[x_1,\hdots,x_n]) \ar[r]  & \text{Spec}(B) \ar[r,hook] &  Y
\end{tikzcd}
\end{center}
since $f$ is locally of finite type, where $\text{Spec}(A)$ and $\text{Spec}(B)$ are open affines in $X$ and $Y$ respectively, and $\text{Spec}(A) \to \text{Spec}(B[x_1,\hdots,x_n])$ is a closed immersion. We then know 
\[  A \cong B[x_1,\hdots,x_n]/\langle g_{1}(\underline{x}),\hdots, g_{k}(\underline{x}) \rangle,  \]
so in model-theoretic terms, we have the corresponding sentence $\Phi_{\text{Spec}(A)}$ in the language $\lorf(B[x_1,\hdots,x_n])$ given explicitly by
\[  \exists x_1,\hdots,x_n:\bigwedge_{i=1}^{k}g_{i}(\underline{x})=0.  \]
$\Phi_{\text{Spec}(A)}$ defines a constructible subspace $K_{\Phi}$ in $\sper B[x_1,\hdots,x_n]$, and for any homomorphism $\varphi:B[x_1,\hdots,x_n] \to \ell$ to a real closed field $\ell$ which satisfies $\Phi_{\text{Spec}(A)}$, we know
\[  \varphi(g_i(\underline{x}))=0  \]
for all $i \in 1,\hdots,k$. We consequently have the factorization 
\begin{center}
\begin{tikzcd}
 B[x_1,\hdots,x_n] \ar[r] \ar[d,two heads] & \ell  \\ A \ar[ru]  &  
\end{tikzcd}
\end{center}
which extends to real closures, so 
\[ K_{\Phi} \cong \sper A. \]
Hence $\rho f$ is locally of finite type. By Corollaries 4.5 and 4.6 of \cite{rcs}, since being \'{e}tale can be checked on affines, $\rho f$ is quasicompact. By Propositions 4.8, 4.9, and 4.15 in \cite{rcs}, 
\[  \rho f\lvert_{\sper A}  \]
is separated for all open affines $\sper A$, so $\rho f$ is quasiseparated. By (\cite{Scheiderer_94},5), $\rho f$ is a local homeomorphism, and if we can show $\rho f$ is a local isomorphism, we will prove $\rho$ is continuous.
\par Since the property 'being \'{e}tale can be checked on affine opens and \'{e}tale morphisms are quasicompact, it suffices to analyze \'{e}tale morphisms $f$ in diagrams of the form
\begin{center}
\begin{tikzcd}
  \text{Spec}(k_{1}\times \hdots \times k_{m}) \ar[r,hook] \ar[d] & \text{Spec}(A) \ar[d,"f"] \\ \text{Spec}(\kappa(\textbf{p})) \ar[r,hook] &  \text{Spec}(B) 
\end{tikzcd}
\end{center}
where $\kappa(\textbf{p})$ is real closed and $k_{i}/\kappa(\textbf{p})$ is a finite separable extension. Since 
\[ \kappa(\textbf{p}) \to k_i \to \overline{\kappa(\textbf{p})}  \]
is a finite extension, $k_i$ is either real closed or algebraically closed, but since algebraically closed fields have empty real closures, we focus on the real closed case. If $k_i$ is real closed, $k_i \cong \kappa(\textbf{p})$. Consider the local subspace $W_{\sper k_i}$ of $\sper A$. Since $f$ has relative dimension 0 and $\rho f$ is surjective by definition of the ret topology, we have an isomorphism of local systems
\[  W_{\sper k_i} \cong W_{\sper \kappa(\textbf{p})} \]
induced by $f$. On a connected open neighborhood $V \subseteq \sper B$ about $\sper \kappa(\textbf{p})$, we take a connected connected open $\rho f^{-1}(V) \cap C_{\sper k_i}$, where $C_{\sper k_i}$ is the connected component of $\sper A$ containing $\sper k_i$. By our isomorphism of local subspaces, 
\[  \rho f^{-1}(V) \cap C_{\sper k_i} \text{ and } V  \]
are equidimensional. We further shrink the neighborhood $\rho f^{-1}(V) \cap C_{\sper k_i}$ to $U$ so that $\rho f\lvert_{U}$ is a homeomorphism onto its image.  
\par To conclude that the local homeomorphism $\rho f$ is a local isomorphism, a lemma about local subspaces is required:
\begin{lemma}
The closed points in an $n$-dimensional real closed space $U$ have isomorphic local subspaces.
\end{lemma}
\begin{proof}
Fix a local subspace $W_{x_1}$ about the closed point $x_1=\sper k_1$, and fix a top dimensional point $\sper L_1 \to W_{x_1}$ so that 
\[  L_1 \cong k_{1}(t_1,\hdots,t_n)  \]
as unordered fields. If $x_2=\sper k_2$ is another closed point of $U$ with top dimensional point $\sper L_2 \to W_{x_2}$, we have 
\[  \text{tr.deg} L_2/k_2 = \text{tr.deg} L_1/k_{1}. \]
I further claim $\text{tr.deg}(k_2/\R_{0})=\text{tr.deg} (k_1/\R_{0})$. Towards a contradiction, assume 
\[ \text{tr.deg} (k_2/\R_{0}) \geq \text{tr.deg} (k_1/\R_{0}),  \]
so that we a diagram of real spectra of fields
\begin{center}
\begin{tikzcd}
\sper L_2 \ar[r] \ar[d] &  \sper k_2 \ar[r] \ar[d] & \sper \R_{0} \\  \sper L_1 \ar[r] & \sper k_{1}. \ar[ru]  &  
\end{tikzcd}
\end{center}
We use the transcendence degree of $k_2/k_1$ to break our analysis of the above diagram into cases. If $\text{tr.deg} (k_2/k_1)=0$, $k_1 \cong k_2$ and $\text{tr.deg} (k_i/\R_{0})$ is a constant function of $i$. Otherwise, we have the equations 
\begin{align*}
\text{tr.deg} (k_2/\R_{0}) =& \text{tr.deg} (k_2/k_1)+\text{tr.deg} (k_1/\R_{0})  \\  \text{tr.deg}(L_i/\R_{0}) =& \text{tr.deg}(L_i/k_i)+\text{tr.deg}(k_i/\R_{0}),
\end{align*}
so if $\text{tr.deg}(k_i/\R_{0})>0$, we have a finite extension
\[ k_{2} \to k_{1} \otimes_{\R_{0}} k_2 \to L_2,  \]
and since 
\[ \text{tr.deg} (k_1 \otimes_{\R_{0}} k_2/k_1) = \text{tr.deg} (k_1/\R_0), \]
both are finite. Properties of the transcendence degree of a field then allow us to conclude we've reached a contradiction, so $\text{tr.deg}(k_2/\R_{0})=\text{tr.deg} (k_1/\R_{0})$ as originally claimed.
\end{proof}
\par Finally, we know homeomorphisms are closed maps, so for any other local system $W_{x_0}$ of a closed point $x_0$ in $U$,
\[ W_{\rho f(x_{1})} \cong W_{x_1} \cong W_{x_0} \cong W_{\rho f(x_{0})}.  \]
Since being a local isomorphism can be checked on local subspaces, and all local subspaces of closed points have been demonstrated to be isomorphic, $\rho f\lvert_{U}$ is an isomorphism onto its image.

\par To prove $\rho$ has property iii., we note that morphisms of real closed spaces in the image of $\rho$ are also scheme morphisms, so property iii. holds for any ret cover.
\par Proving $\rho$ has property iv. is very similar to proving $\rho$ is cocontinuous! For $\sper B_{\lambda}$ an affine open cover of $\rho U$, and $\sper A_{\kappa}$ an open affine cover of $\rho U'\times_{\rho U} \sper B_{\lambda}$, the desired cover is $\text{Spec}(A_{\kappa}) \to U'$ and the desired collection of morphisms is $\text{Spec}(A_{\kappa}) \to U$.
\par To prove the final property, we note that a real closed space $Y$ (over $\rho X$) has a $\rce$-cover $\rho U_{i}$ in which each $U_i$ is \'{e}tale over $X$. We take an open affine cover of $Y$ by real spectra $\sper B_{\lambda}$, and this cover satisfies condition v. This concludes the proof. 
\end{proof}

\subsection{Simplicial Sheaves}

We need to enlarge the category so that real closed spaces and simplices are objects in the same category, in order to construct our own homotopy theory. 

We make use of Nisnevich squares in forming homotopy limits and colimits, so we take a deep look at those squares below; 

\begin{defn}
Let 
\begin{center}
\begin{tikzcd}
U \times_{X} V \ar[r] \ar[d] & V \ar[d,"p"] \\ U \ar[r,hook] & X
\end{tikzcd}
\end{center} 
be a commutative pullback square. This square is an \textit{elementary distinguished Nisnevich square} if $U \to X$ is an open immersion, $p$ is real \'{e}tale, and $p^{-1}(X \setminus U) \to X \setminus U$ is an isomorphism. 
\end{defn}

\begin{ex}
Let $X=\sper \R[u]$. Let $U$ be given by the sentence 
\[  \exists u:u <1 \]
in the language $\lorf(\R[u])$, and let $V$ be given by
\[  \exists u:u >0. \]
I claim 
\begin{center}
\begin{tikzcd}
U \times_{X} V \ar[r] \ar[d] & V \ar[d,hook] \\ U \ar[r,hook] & X
\end{tikzcd}
\end{center}
is an elementary distinguished Nisnevich square. We know $U$ is an open subspace of $X$ and $V \hookrightarrow X$ is real \'{e}tale. To analyze the complement of $U$, we consider the diagram 
\begin{center}
\begin{tikzcd}
& V \ar[d,hook] \\ X \setminus U \ar[r,hook] & X.
\end{tikzcd}
\end{center}
The morphisms from both $X \setminus U$ and $V$ to $X$ are inclusions, so the limit over the above diagram is an intersection, and since $X \setminus U \subset V$,
\[   p^{-1}(X \setminus U) \cong X \setminus U \]
and we have demonstrated what we set out to show.
\end{ex}

\begin{remark}
We can generalize the previous example; if $p:V \to X$ is a real \'{e}tale morphism whose image contains $X \setminus U$, we can form a diagram
\begin{center}
\begin{tikzcd}
U \times_{X} V \ar[dd] \ar[r] & V \ar[d,"p"] \\    &  p(V) \ar[d,hook] \\  U \ar[r,hook] & X.
\end{tikzcd}
\end{center}
Since $p$ is quasicompact, $p(V)$ has its own real closed space structure and is locally constructible in $X$. When we form the fiber product $p^{-1}(X \setminus U)$, we may form it in two stages;
\begin{center}
\begin{tikzcd}
 V \times_{X} (X \setminus U) \ar[r,hook] \ar[d,"p_{X \setminus U}" swap] & V \ar[d,"p"] \\ X \setminus U \ar[d,hook] \ar[r,hook] & p(V) \ar[d,hook] \\ X \setminus U \ar[r,hook] & X,
\end{tikzcd}
\end{center}
where 
\[  V \times_{X} (X \setminus U) \cong p^{-1}(X \setminus U).\]
The previous square is an elementary distinuished Nisnevich square if $p_{X \setminus U}$ is an isomorphism.
For example, if $X=\sper(\R[x]) \setminus \{-1,1\}$, let $U=X \setminus \{0\}$, and let $V \subset \sper \R[x,y]$ be given by 
\[  \begin{cases} x=(y-1)^{2} & \text{ if } x>1 \\ y=0 & \text{ if } -1<x<1 \\ x=-(y+1)^{2} & \text{ if } x<-1.  \end{cases} \]
The projection to the $x$-coordinate is a real \'{e}tale morphism $V \to X$, and 
\begin{center}
\begin{tikzcd}
p^{-1}(U) \ar[r,hook] \ar[d]  & V \ar[d] \\ U \ar[r,hook] & X
\end{tikzcd}
\end{center}
is an elementary distinguished Nisnevich square.
\end{remark}

Now we form the simplicial model category central to our study, which is a category large enough to contain an $n$-simplex for each $n$, as well as every real closed space over $X$;

\begin{defn}
Let $\Delta$ be the category whose objects are finite totally ordered sets $[n]=(\{0,1,2,\hdots, n\},\leq)$, with order-preserving functions as our morphisms. Let $\text{PreSh}(\rce /X)$ be the category of presheaves on the real \'{e}tale site of $X$. A \textit{simplicial presheaf on} $X$ is a functor $\mathscr{F}:\Delta^{\text{op}} \to \text{PreSh}(\rce /X)$. A simplicial presheaf is a \textit{simplicial sheaf} if, for each $[n]$ and real \'{e}tale cover $\{ U_{\lambda} \}_{\lambda \in \Lambda}$, the sequence 
\begin{center}
\begin{tikzcd}
  \mathscr{F}_{n}(X) \ar[r] & {\displaystyle \prod_{\lambda \in \Lambda}\mathscr{F}_{n}(U_{\lambda})} \ar[r,bend left=10] \ar[r,bend right=10] &  {\displaystyle \prod_{\kappa,\lambda \in \Lambda} \mathscr{F}_{n}(U_{\kappa} \times_{X} U_{\lambda}) }
\end{tikzcd}
\end{center}
is an equalizer. We denote the collection of these sheaves $\mathscr{F}$ by $\Delta^{op}\text{Sh}(\rce /X)$. 
\end{defn}

\begin{ex}
Fix a real closed space $Y$. Consider the functor $\text{Hom}_{\rce /X}(-,Y)$ from $\rce /X$ to $\underline{Sets}$, and define \[  \mathscr{F}_{n}:=\text{Hom}_{\rce /X}(-,Y).\]
I claim $\mathscr{F}$ defines a simplicial sheaf on $X$, where each morphism in $\Delta$ is sent to the identity natural transformation $\text{Hom}_{\rce /X}(-,Y) \to \text{Hom}_{\rce /X}(-,Y)$. Since real closed spaces form a locally small category, $\text{Hom}_{\rce /X}(-,Y)$ is a set when evaluated on any real closed space $V$, so $\mathscr{F}$ is a simplicial presheaf. To prove $\mathscr{F}$ is a simplicial sheaf, let $\{ U_{\lambda} \}_{\lambda \in \Lambda}$ be a real \'{e}tale cover of $X$, and let $n$ be nonnegative. Given a section \[ (\sigma_{\lambda})_{\lambda \in \Lambda} \in  {\displaystyle \prod_{\lambda \in \Lambda}\mathscr{F}_{n}(U_{\lambda})} \] such that for each pair $\lambda, \kappa \in \Lambda$,
\[  \sigma_{\lambda} \biggr\lvert_{U_{\kappa} \times_{X} U_{\lambda}}=\sigma_{\kappa} \biggr\lvert_{U_{\kappa} \times_{X} U_{\lambda}}. \]
We can give a function $\sigma \in \mathscr{F}_{n}(X)=\text{Hom}_{\rce /X}(X,Y)$ which restricts to $\sigma_{\lambda}$ on $U_{\lambda}$, proving sections of $\mathscr{F}_n$ glue. To show that this gluing is unique, we map the singleton 'test set' $\{t\}$ to our diagram;
\begin{center}
\begin{tikzcd}
  \mathscr{F}_{n}(X) \ar[r] & {\displaystyle \prod_{\lambda \in \Lambda}\mathscr{F}_{n}(U_{\lambda})} \ar[r,bend left=10] \ar[r,bend right=10] &  {\displaystyle \prod_{\kappa,\lambda \in \Lambda} \mathscr{F}_{n}(U_{\kappa} \times_{X} U_{\lambda}) } \\ \{t\}. \ar[ru] &  &
\end{tikzcd}
\end{center}
Glueing functions of sets gives a factorization 
\begin{center}
\begin{tikzcd}
  \mathscr{F}_{n}(X) \ar[r] & {\displaystyle \prod_{\lambda \in \Lambda}\mathscr{F}_{n}(U_{\lambda})} \ar[r,bend left=10] \ar[r,bend right=10] &  {\displaystyle \prod_{\kappa,\lambda \in \Lambda} \mathscr{F}_{n}(U_{\kappa} \times_{X} U_{\lambda}) } \\ \{t\}, \ar[u,dotted,"\exists"] \ar[ru] &  &
\end{tikzcd}
\end{center}
and this factorization is unique since functions $\text{Hom}_{\rce /X}(X,Y)$ are well-defined by their restrictions to the opens in our real \'{e}tale cover. 
\end{ex}

For any morphism $f:X \to Y$ of real closed spaces, there is a  functor 
\[  f_{\ast}:\text{Sh}(\rce /X) \to \text{Sh}(\rce /Y), \]
which takes a sheaf $\mathscr{F}$ on $X$ and outputs a sheaf $f_{\ast}\mathscr{F}$ on $Y$. On a covering open $V \to Y$ of $Y$, 
\[   f_{\ast}\mathscr{F}(V):=\mathscr{F}(f^{-1}(V)),   \]
where 
\begin{center}
\begin{tikzcd}
f^{-1}(V) \ar[r] \ar[d] & X \ar[d,"f"] \\ V \ar[r] & Y
\end{tikzcd}
\end{center}
is a fibered square. Consequently, we have a morphism
\[  \text{Hom}_{\rce /Y}(-,Y) \to f_{\ast}\text{Hom}_{\rce /X}(-,X)    \]
of sheaves, given by taking the real closed \'{e}tale map $V \to Y$ to the real closed \'{e}tale map $f^{-1}(V) \to X$.

\begin{prop}
A morphism $f:\mathscr{X} \to \mathscr{Y}$ in $\Delta^{\text{op}}\text{Sh}(\rce/X)$ is a monomorphism if and only if, for each $n \geq 0$ and each $\rce$-morphism $U \to X$, the set map
\[  f_{n,U}:\mathscr{X}_{n}(U) \to \mathscr{Y}_{n}(U)  \]
is a monomorphism.
\end{prop}
\begin{proof}
Let 
\begin{center}
\begin{tikzcd}
\mathscr{T} \ar[r,"g^{1}", bend left=30] \ar[r,"g^{2}", bend right=30,swap] & \mathscr{X} \ar[r,"f"]  & \mathscr{Y}
\end{tikzcd}    
\end{center}
be a commutative diagram, i.e. $f \circ g^{1}=f \circ g^2$. If $f$ is a monomorphism, $g^{1} =g^{2}$ by definition, so 
\[ g_{n,U}^{1}=g_{n,U}^{2}  \]
for all $U$ and $n$. We have shown $f_{n,U}$ is a monomorphism for all choices of $n$ and $U$. 
\par On the other hand, suppose $f_{n,U}$ is a monomorphism for any choice of $n$ and $U$. We then know $g_{n,U}^{1}=g_{n,U}^{2}$ for all $n$ and $U$. On a $\rce$ cover $\{U_{\lambda}\}_{\lambda \in \Lambda}$ of $X$, $g^{i}$ is a sheaf morphism, so for $\beta, \lambda \in \Lambda$
\[ g_{n,U_{\lambda}}^{1}\lvert_{U_{\lambda} \times_{X} U_{\beta}}=g_{n,U_{\beta}}^{1}\lvert_{U_{\lambda} \times_{X} U_{\beta}}.  \]
Hence the sections $g_{n,U_{\lambda}}^{i}$ glue to form a sheaf morphism $g_{n}^{i}$, and since these morphisms are equal on each open $U_\lambda$, 
\[  g_{n}^{1}=g_{n}^{2}  \]
for all $n$, proving our claim.
\end{proof}

\subsection{The Simplicial Model Category Structure}

\begin{defn}
A morphism $f:\mathcal{X} \to \mathcal{Y}$ of simplicial sheaves on the site $(\rce /T,\tau)$ is a \textit{weak equivalence} if for any point $x:\text{Sh}(\rce /T) \to \underline{Sets}$ of $(\rce /T,\tau)$, $x(f)$ is a weak equivalence of simplicial sets (i.e. a homotopy equivalence on geometric realizations). 
\end{defn}

\begin{thm}
We take the class of monomorphisms in $\Delta^{\text{op}}\text{Sh}(\rce /T)$ to be \textit{cofibrations}. The class of morphisms which have the right lifting property with respect to any cofibration which is also a weak equivalence are called \textit{fibrations}. The classes $(\textbf{W},\textbf{C},\textbf{F})$ give the category $\Delta^{\text{op}}\text{Sh}(\rce /T)$ the structure of a model category
\end{thm}
\begin{proof}
This result follows directly from \cite{jardine}, since they prove these classes of morphisms give a simplicial model structure to the simplicial sheaves on \textit{any} Grothendieck site. 
\par To give a bit more detail, proving these three classes of morphisms satisfy all the axioms of a model category \underline{except} that each morphism $f:\mathscr{X} \to \mathscr{Y}$ has a factorization into a trivial cofibration followed by a fibration, is straightforward. To prove the 5th axiom of a model category, Jardine uses two techniques which are too complicated to fully explain here; Boolean Localization and The Bounded Cofibration Condition.  Boolean localization is used to show that we can check whether a morphism is a local weak equivalence on a Boolean Algebra instead of the given site. The Bounded Cofibration Condition is a theorem, which says as long as our site $\rce /T$ doesn't have too many morphisms, morphisms from a 'small' object can be factored into a monomorphism and a trivial cofibration.
\end{proof}

We may Bousfield localize this model structure with respect to the class of morphisms \[\{f:[0,1]_{k} \times_{T} X \to X\}_{k=\rho k,X \to T \text{ is } \rce} \] to form the homotopy category of simplicial sheaves. 


\section{The Future: Understanding Representable Cohomology on RCS}

I suspect we can further contextualize the result of \cite{no_deRham_huber} in the homotopy category formed above that there is no de Rham cohomology for semialgebraic spaces using the motivic homotopy theory and resulting representable cohomology theories developed here. It is unclear whether or not the theory of motives is necessary to contextualize Professor Huber's work, but I hope to investigate further in the future. I also suspect there is a version of gradient descent, using 'local subspaces like tangent bundles', which can be done in real closed spaces, especially if those real closed spaces are semialgebraic.

\bibliographystyle{plain}
\bibliography{citations_RCScohomology}

\end{document}